\documentclass[psamsfonts,12pt]{amsart}
\usepackage[T1]{fontenc}
\usepackage{charter}
\usepackage[margin=1in]{geometry}  % set the margins to 1in on all sides
\usepackage{graphicx}              % to include figures
\usepackage{amsmath}            % great math stuff
\usepackage{amsfonts}              % for blackboard bold, etc
\usepackage{amsthm}                % better theorem environments
\usepackage{verbatim}              % to comment out large swathes of text with the comment environment
\usepackage{indentfirst}           % to indent the first paragraph of a section

\usepackage[shortlabels]{enumitem}
\usepackage{mathtools} 
\usepackage{amssymb}
\usepackage[all]{xy}
\usepackage{caption}
\usepackage{subcaption}
\usepackage{setspace}
\usepackage{mathrsfs}
\usepackage{pinlabel}
\usepackage[percent]{overpic}
\usepackage{float}
\usepackage{longtable}
\usepackage[colorlinks=false, hidelinks]{hyperref}
\usepackage[none]{hyphenat}
\usepackage{rotating}
\allowdisplaybreaks[4]
\makeatletter
\@namedef{subjclassname@2020}{\textup{2020} Mathematics Subject Classification}
\makeatother

\DeclareMathAlphabet{\mathpzc}{OT1}{pzc}{m}{it}

\usepackage{tikz}
\usepackage{tkz-euclide} 
\tikzset{hidden/.style = {thick, dashed}}
\usetikzlibrary{plotmarks}
\usetikzlibrary{arrows}

\newtheorem{thm}{Theorem}[section]
\newtheorem{lem}[thm]{Lemma}
\newtheorem{prop}[thm]{Proposition}
\newtheorem*{definition*}{Definition}
\newtheorem{cor}[thm]{Corollary}
\newtheorem{conj}[thm]{Conjecture}

\theoremstyle{remark}
\newtheorem*{rmk}{Remark}
\theoremstyle{remark}

\theoremstyle{definition}\newtheorem*{keyfact}{Key Fact}

\newcommand{\verteq}[0]{\begin{turn}{90} $=$\end{turn}}

\newcommand{\C}{\mathbb{C}}
\newcommand{\fr}{\frac}

\newcommand{\CC}{\mathbb{C}}
\newcommand{\ZZ}{\mathbb{Z}}
     
\newcommand{\QQ}{\mathbb{Q}}

\usepackage{amssymb,fge}

\newcommand{\Mcal}{{\mathcal M}}

\newcommand{\Zcal}{{\mathcal Z}}
\begin{document} \title[The Uniform Boundedness and Dynamical Lang Conjectures]{The Uniform Boundedness and Dynamical Lang Conjectures for polynomials}
	\author[N.~R.~Looper]{Nicole R. Looper}
	\email{nicole\_looper@brown.edu}
	 \thanks{The author's research was supported by NSF grant DMS-1803021.}

	\begin{abstract} We give a conditional proof of the Uniform Boundedness Conjecture of Morton and Silverman in the case of polynomials over number fields, assuming a standard conjecture in arithmetic geometry. Our technique simultaneously yields a dynamical analogue of Lang's conjecture on minimal canonical heights for these maps. We obtain similar results for non-isotrivial polynomials over a function field of characteristic zero. When the latter are unicritical of degree at least 5, the results hold unconditionally.
	\end{abstract}
	\maketitle
	\bigskip
	\section{Introduction}
	\label{section:introduction}
	
	In the dynamics of rational functions $f:\mathbb{P}^1\to\mathbb{P}^1$ of degree $d\ge2$ defined over a number field $K$, two conjectures stipulate that few points of $\mathbb{P}^1(K)$ have small canonical height $\hat{h}_f$ relative to $f$, in a way that depends only on $d$ and $K$. The first of these conjectures is the Uniform Boundedness Conjecture of Morton and Silverman \cite{MortonSilverman}. A dynamical generalization of Merel's theorem on the torsion points of an elliptic curve over a number field, this conjecture can be shown to imply the well-known Torsion Conjecture on abelian varieties \cite{Fakhruddin}.
	
	\begin{conj}[Uniform Boundedness Conjecture \cite{MortonSilverman}]\label{conj:UBC} Let $d\ge2$, $N\ge1$, and let $K$ be a number field. Let $f:\mathbb{P}^N\to\mathbb{P}^N$ be a morphism of degree $d$ defined over $K$. There is a constant $B=B(d,N,[K:\mathbb{Q}])$ such that $f$ has at most $B$ preperiodic points in $\mathbb{P}^N(K)$.\end{conj}
	
	Aside from the special case of Latt\`{e}s maps and the unicritical maps studied in \cite{Looper:UBCunicrit,Panraksa}, progress on Conjecture \ref{conj:UBC} has only been obtained by imposing strong local conditions on the dynamics of $f$, as in \cite{Baker:dynGreen, Benedetto2, Hutz}. We will forfeit all such local hypotheses, at the expense of assuming a standard conjecture in arithmetic geometry, in order to prove Conjecture \ref{conj:UBC} for the $K$-rational preperiodic points of polynomials.
	
	Our first result is the following.
	\newline
	
	\begin{thm}\label{thm:UBCpolys}Let $K$ be a number field (resp.~a one-dimensional function field of characteristic zero), and let $d\ge2$. Assume the $abcd$-conjecture (Conjecture \ref{conj:abcd}) for $K$. Then there is a constant $B=B(d,K)$ with the property that if $f\in K[z]$ is a polynomial (resp.~ a non-isotrivial polynomial) of degree $d$, then $f$ has at most $B$ preperiodic points contained in $K$.\end{thm}
	
	The second conjecture concerning uniform bounds on points of small canonical height is the Dynamical Lang Conjecture, proposed by Silverman \cite[Conjecture 4.98]{Silverman:ADS}. The technique used to prove Theorem \ref{thm:UBCpolys} allows us to prove a weaker version of this conjecture in the case of polynomial maps.
	
	\begin{thm}\label{thm:dynLangpolys} Let $K$ be a number field or a one-dimensional function field of characteristic zero, and let $d\ge2$. Assume the $abcd$-conjecture (Conjecture \ref{conj:abcd}) for $K$. Then there is a $\kappa=\kappa(d,K)>0$ such that for any degree $d$ polynomial $f\in K[z]$ with critical height $h_{\textup{crit}}(f)$, and for any $P\in K$, either $\hat{h}_f(P)=0$, or \begin{equation}\label{eqn:lbLang}\hat{h}_f(P)\ge\kappa\max\{1,h_{\textup{crit}}(f)\}.\end{equation} \end{thm}
	
	The critical height is a dynamically natural measurement of the height of $f$, defined in Section \ref{subsec:potthy}. When $K$ is a number field, the critical height $h_{\textup{crit}}(f)$ is commensurate to the Weil height of $f$ as a point in the moduli space $\mathcal{M}_d$ of degree $d$ endomorphisms of $\mathbb{P}^1$ \cite[Theorem 1]{Ingram:critmod}, and by Northcott's Theorem (\ref{eqn:lbLang}) becomes $\hat{h}_f(P) \geq \kappa \max\{ 1, h_{\Mcal_d}(f)\}$
	for this moduli height $h_{\Mcal_d}(f)$ and some appropriate modification of the constant $\kappa$. This lines up with the formulation of the conjecture given in \cite[Conjecture 4.98]{Silverman:ADS}.
	
	We remark that the $abcd$-conjecture is a consequence of Vojta's conjecture with truncated counting function, as is shown in the discussion after \cite[Conjecture 2.2]{Looper:UBCunicrit}. The \linebreak$abcd$-conjecture can be thought of as generalizing the $abc$-conjecture to higher dimensions, or alternatively, to linear Diophantine equations with more than two summands. Vojta's conjecture with truncated counting function implies what is usually referred to as Vojta's conjecture \emph{tout court}. This weaker version already has Schmidt's subspace theorem as a special case.
	
	The strategy for proving Theorems \ref{thm:UBCpolys} and \ref{thm:dynLangpolys} is as follows. First we prove a global quantitative equidistribution result (Theorem \ref{thm:globaleq}) that is uniform across all degree $d$ polynomials. Specifically, we introduce a geometric notion quantifying a certain kind of equidistribution at each place of bad reduction of $f\in K[z]$, and then prove an upper bound on the average pairwise logarithmic distance between points of small local canonical height realizing a given failure of equidistribution. This upper bound is uniform across places of bad reduction, and is also uniform across degree $d$ polynomials over $K$. We derive from this a theorem stating that large sets of points having small height must be roughly equidistributed at ``most'' places of bad reduction. The implied constants are again independent of the degree $d$ map $f\in K[z]$. This theorem upgrades \cite[Corollary 3.8]{Looper:UBCunicrit} in passing from unicritical to arbitrary polynomials, as well as from preperiodic points to points of small height.
	
	Theorem \ref{thm:globaleq} implies that differences of points of small height typically have most of their prime support contained within the set of places of bad reduction of $f$. This phenomenon is articulated precisely in Proposition \ref{prop:adgoodbars}. The proof of Theorems \ref{thm:UBCpolys} and \ref{thm:dynLangpolys}, appearing in \S\ref{section:proof}, uses this fact, combined with the geometric descriptions of these point configurations in the local filled Julia sets, to derive a contradiction of the $abcd$-conjecture (Conjecture \ref{conj:abcd}) when too many of these points are assumed to lie in $K$. The contradiction is obtained by considering the prime factorizations of cross-ratios of quadruples of $K$-points of small height, and then using the Grassmann--Pl\"ucker relations satisfied by these cross-ratios to furnish an $abcd$-tuple.
	
	Quantitative equidistribution theorems have served as a tool in studying points of small canonical height across families of dynamical systems; see for example \cite{BakerIhRumely, BakerPetsche, DKY:unicrit, DKY:Lattes}. Many of these have, though not always explicitly, been formulated in terms of the energy pairing studied in \cite{Fili,PST}, and had \cite[Th\'eor\`eme 3]{FRL} as their underlying substrate. On the other hand, the geometric information leveraged in the proof of Theorems \ref{thm:UBCpolys} and \ref{thm:dynLangpolys} does not appear to be readily accessible via equidistribution results given solely in terms of energy pairings. Theorem \ref{thm:globaleq} presents a formulation that is well suited both to the main theorems of this article, and to those of \cite{Looper:Bogomolov}.
	
	Finally, we note that Theorem \ref{thm:globaleq}, along with its consequence Proposition \ref{prop:adgoodbars}, can be directly combined with \cite[\S7]{Looper:UBCunicrit} (with triangles in lieu of quadrilaterals) to prove the following.
	
	\begin{thm}\label{thm:dynLangunicrit} Let $K$ be a number field or a one-dimensional function field of characteristic zero, and let $d\ge5$. If $K$ is a number field, assume the $abc$-conjecture for $K$. There is a $\kappa=\kappa(d,K)>0$ with the following property. Let $f(z)=z^d+c\in K[z]$, and assume $c$ is not in the field of constants of $K$ if $K$ is a function field. Then for all $P\in K$, either $\hat{h}_f(P)=0$, or \[\hat{h}_f(P)\ge\kappa\max\{1,h_{\textup{crit}}(f)\}.\]\end{thm} In particular, this result holds unconditionally when $K$ is a function field. The proof is immediate by adapting \cite[\S7--8]{Looper:UBCunicrit} to incorporate Proposition \ref{prop:adgoodbars}.
	
	\indent\textbf{Acknowledgements.} I would like to thank Rob Benedetto, Laura DeMarco, Holly Krieger, Joe Silverman, and Tom Tucker for helpful discussions related to this project. I thank Rob Benedetto in particular for suggesting the use of cross ratios as a potential approach in this problem. I also thank Matt Baker for useful suggestions concerning the exposition. Finally, I would like to sincerely thank the referees (particularly one extraordinarily diligent referee) for an extremely careful and astute job, and for a very great number of helpful comments on several drafts of this paper.
	
	\section{Background}\label{section:background}
	
	\subsection{Notation} We set the following notation:
	
	\setlength{\tabcolsep}{15pt}
	
	\begin{tabular}{r p{10cm}}
		
		$K$ & a number field, or a finite extension of a field $k(t)$ of rational functions in one variable over a field $k$ of characteristic $0$ or characteristic $>d$, where $k$ is assumed without loss to be algebraically closed  \\ $F$ & $\QQ$ if $K$ is a number field, or $k(t)$ if $K/k(t)$ is a function field \\ $M_K$ & a complete set of inequivalent places of $K$, with absolute values $|\cdot|_v$ normalized to extend the standard absolute values on $F$ \\ $M_K^0$ & the set of non-archimedean places in $M_K$  \\ $M_K^\infty$ & the set of archimedean places in $M_K$\\ $\mathscr{S}_d$ & the set of places of $K$ whose residue characteristic is $\le d$ (empty if $K$ is a function field) \\ $k_{\mathfrak{p}}$ & the residue field associated to the non-archimedean place $\mathfrak{p}$ of $K$ \\
		$N_\mathfrak{p}$ & $\fr{\log(\#k_\mathfrak{p})}{[K:\QQ]}$ for $K$ a number field and $\mathfrak{p}\in M_K^0$ \\ & $\fr{[k_\mathfrak{p}:k]}{[K:k(t)]}$ for $K/k(t)$ a function field and $\mathfrak{p}\in M_K$\\ $r_v$ & $\frac{[K_v:F_v]}{[K:F]}$
	\end{tabular} 
	
	\noindent If $K$ is a number field, let $\mathcal{O}_K$ denote the ring of integers. If $K$ is a function field, let $\mathcal{O}_K$ be the integral closure of $k[t]$ in $K$. If $K$ is a number field, $n\ge 2$ and $P=(z_1,\dots,z_n)\in\mathbb{P}^{n-1}(K)$ with $z_1,\dots,z_n\in K$, let \begin{equation*}\begin{split}h(P)=&\sum_{\textup{primes }\mathfrak{p} \textup{ of } \mathcal{O}_K} -\min\{v_{\mathfrak{p}}(z_1),\dots,v_{\mathfrak{p}}(z_n)\}N_{\mathfrak{p}}\\&+\dfrac{1}{[K:\QQ]} \sum_{\sigma:K\hookrightarrow\CC} \log\max\{|\sigma(z_1)|,\dots,|\sigma(z_n)|\},\end{split}\end{equation*} where we do not identify conjugate embeddings. Here, $v_\mathfrak{p}$ is normalized to have $\mathbb{Z}$ as its image. If $K$ is a function field, let \[h(P)=\sum_{\textup{primes }\mathfrak{p} \textup{ of } \mathcal{O}_K} -\min\{v_{\mathfrak{p}}(z_1),\dots,v_{\mathfrak{p}}(z_n)\}N_{\mathfrak{p}}.\] For any $P=(z_1,\dots,z_n)\in\mathbb{P}^{n-1}(K)$ with $z_1,\dots,z_n\in K^*$, define \begin{equation}\label{eqn:I} I(P)=\{\textup{primes }\mathfrak{p} \textup{ of } \mathcal{O}_K\mid v_{\mathfrak{p}}(z_i)\ne v_{\mathfrak{p}}(z_j)\textup{ for some } 1\le i,j\le n\}\end{equation} and let \begin{equation*}\textup{rad}(P)=\sum_{\mathfrak{p}\in I(P)} N_{\mathfrak{p}}.\end{equation*} \subsection{The $abcd$-conjecture} In order to prove Theorems \ref{thm:UBCpolys} and \ref{thm:dynLangpolys}, we will use a generalization of the $abc$-conjecture. The standard $abc$-conjecture corresponds to the case $n=3$ in Conjecture \ref{conj:abcd} below. 
	
	\begin{conj}[The $abcd$-conjecture]{\label{conj:abcd}} Let $K$ be a number field or a one-dimensional function field of characteristic zero, and let $n\ge 3$. Let $[Z_1:\cdots:Z_n]$ be the standard homogeneous coordinates on $\mathbb{P}^{n-1}(K)$, and let $\mathcal{H}$ be the hyperplane given by $Z_1+\dots+Z_n=0$. For any $\epsilon>0$, there is a proper Zariski closed subset $\mathcal{Z}=\mathcal{Z}(K,\epsilon,n)\subsetneq\mathcal{H}$ and a constant $C_{K,\mathcal{Z},\epsilon,n}$ such that for all $P=(z_1,\dots,z_n)\in\mathcal{H}\setminus\mathcal{Z}$ with $z_1,\dots,z_n\in K^*$, we have \[h(P)<(1+\epsilon)\textup{rad}(P)+C_{K,\mathcal{Z},\epsilon,n}.\]\end{conj}
	
	For $X$ a smooth projective variety, a divisor $D\in\text{Div}(X)$ and $v\in M_K$, let $\lambda_{D,v}$ be a $v$-adic local height on $(X\setminus D)(K_v)$ relative to $D$. (For background on local height functions, see \cite[Chapter B.8]{HindrySilverman:DiophantineGeometry}.) For $P\in X(K)$, let $h_D(P)=\sum_{v\in M_K}r_v\lambda_{D,v}(P)$. We say that an effective divisor $D\in\textup{Div}(X)$ is a \emph{normal crossings divisor} if $D=\sum_{i=1}^r D_i$ for distinct irreducible subvarieties $D_i$, and the variety $\cup_{i=1}^r D_i$ has normal crossings.
	
	\begin{definition*} Let $S\subseteq M_K$ be a finite set of places of $K$ containing $M_K^\infty$, let $X$ be a smooth projective variety, and let $D$ be an effective divisor on $X$. For $P\in X(K)\setminus D$, and $\lambda_{D,\mathfrak{p}}$ a set of local height functions relative to $D$, the arithmetic truncated counting function is \[N_S^{(1)}(D,P)=\sum_{\mathfrak{p}\in M_K\setminus S}
		\chi(\lambda_{D,\mathfrak{p}}(P))N_\mathfrak{p}\] where for $a\in\mathbb{R}$, \[\chi(a)=\begin{cases}0 & \text{if }a\le0\\ 1 & \text{if }a>0.\end{cases}\]\end{definition*} 
	
	In the discussion succeeding \cite[Conjecture 2.2]{Looper:UBCunicrit}, it is shown that Conjecture \ref{conj:abcd} is a consequence of the following conjecture \cite[Conjecture 2.3]{Vojta} of Vojta.
	
	\begin{conj}\cite[Conjecture 2.3]{Vojta}\label{conj:Vojta} Let $K$ be a number field or a one-dimensional function field of characteristic zero, and let $S$ be a finite set of places of $K$ containing the archimedean places. Let $X$ be a smooth projective variety over $K$, let $D$ be a normal crossings divisor on $X$, let $K_X$ be a canonical divisor on $X$, let $A$ be an ample divisor on $X$, and let $\epsilon>0$. Then there exists a proper Zariski closed subset $\mathcal{Z}=\mathcal{Z}(K,S,X,D,A,\epsilon)\subsetneq X$ such that \begin{equation*}N_S^{(1)}(D,P)\ge h_{K_X+D}(P)-\epsilon h_A(P)+O(1)\end{equation*} for all $P\in X(K)\setminus\mathcal{Z}$.
	\end{conj}
	
	\begin{rmk}The version of this conjecture appearing in \cite[Conjecture 2.3]{Vojta} is stated for points $P\in X(L)$ where $L$ has bounded degree over $K$, at the expense of a logarithmic discriminant term $d(L/K)$. Here we will only require the weaker statement appearing in Conjecture \ref{conj:Vojta}.\end{rmk}
	
	\subsection{Non-archimedean potential theory}\label{subsec:potthy} For $v\in M_K$, let $\CC_v$ be the $v$-adic completion of $\overline{K_v}$. We recall that the absolute values $|\cdot|_v$ are normalized to extend the standard $v$-adic absolute values on $F$ (satisfying $|p|_p=\frac{1}{p}$ for all primes $p\in\mathbb{Z}_+$ when $K$ is a number field, and $|t-a|_v=\frac{1}{e}$ for $v$ corresponding to the prime $(t-a)\mathcal{O}_K$ when $K$ is a function field). For $f(z)\in\CC_v[z]$ of degree $d\ge2$ and $z\in\C_v$, let \[\hat{\lambda}_v(z)=\lim_{n\to\infty} \fr{1}{d^n}\log\max\{1,|f^n(z)|_v\}\] be the standard $v$-adic escape-rate function. (See \cite[\S3.4, 3.5]{Silverman:ADS} for a proof that the limit defining $\hat{\lambda}_v(z)$ exists.) Note that $\hat{\lambda}_v(z)$ obeys the transformation rule \begin{equation*}\hat{\lambda}_v(f(z))=d\hat{\lambda}_v(z)\end{equation*} for all $z\in \CC_v$. 
	
	Let $\mathbf{A}_v^1$ denote the Berkovich affine line over $\CC_v$, where each point $x\in\mathbf{A}_v^1$ is associated to a seminorm $[\cdot]_x$ on $\mathbb{C}_v[T]$. For $a\in\CC_v$, we define open and closed Berkovich disks of radius $r$ \[\mathcal{B}(a,r)^-=\{x\in\textbf{A}_v^1:[T-a]_x<r\}\] and \[\mathcal{B}(a,r)=\{x\in\textbf{A}_v^1:[T-a]_x\le r\},\] corresponding to the classical disks \[D(a,r)^-=\{z\in\CC_v:|z-a|_v< r\}\] and  \[D(a,r)=\{z\in\CC_v:|z-a|_v\le r\}\] respectively. A basis for the open sets of $\textbf{A}_v^1$ is given by sets of the form $\mathcal{B}(a,r)^-$ and $\mathcal{B}(a,r)^-\setminus\cup_{i=1}^N\mathcal{B}(a_i,r_i)$, where $a,a_i\in \CC_v$ and $r,r_i>0$. We consider $\textbf{A}_v^1$ as a measure space whose Borel $\sigma$-algebra is generated by this topology. For $z,w\in\mathbf{A}_v^1$, let $\delta_v(z,w)$ denote the Hsia kernel relative to infinity (see \cite[Section 4.1]{BakerRumely}); it is the radius of the smallest Berkovich disk containing both $z$ and $w$. For Berkovich sets $\mathcal{B}_1,\mathcal{B}_2$, we will write \[\delta_v(\mathcal{B}_1,\mathcal{B}_2)=\sup_{x\in\mathcal{B}_1,y\in\mathcal{B}_2}\{\delta_v(x,y)\}\] for convenience. The Berkovich $v$-adic filled Julia set of $f(z)\in\mathbb{C}_v[z]$ is defined as \[\mathcal{K}_v=\bigcup_{M>0}\{x\in\textbf{A}_v^1:[f^n(z)]_x\le M\text{ for all } n\ge0\}.\]
	
	Now suppose $|\cdot|_v$ is non-archimedean. By \cite[Proposition 7.33]{Benedetto:book}, if $\mathcal{E}$ is the smallest disk containing $\mathcal{K}_v$, then for all $m\ge 1$, $f^{-m}(\mathcal{E})$ is a finite union of disjoint closed disks. We refer to the preimage disks as the \emph{disk components} of $f^{-m}(\mathcal{E})$. If $f^{-1}(\mathcal{E})$ is not a disk, we say that $v$ is a place of \emph{bad reduction}, or alternatively a \emph{bad place}. For a bad place $v$, the log of the radius of $\mathcal{E}$ will be called the \emph{splitting radius} of $f$ at $v$. It will be denoted $g_v$.
	
	If $R_f$ is the set of finite critical points of $f$, we will define the \emph{$v$-adic critical height} to be \[\lambda_{\textup{crit},v}(f):=\max_{a\in R_f}\{\hat{\lambda}_v(a)\}.\] For $f(z)\in K[z]$, let \[h_{\textup{crit}}(f)=\sum_{v\in M_K}r_v\lambda_{\textup{crit},v}(f).\] When $v\in M_K^0\setminus\mathscr{S}_d$ is a place of bad reduction for a monic polynomial $f(z)\in K[z]$ of degree $d$, the splitting radius $g_v$ satisfies  \begin{equation}\label{eqn:splittingradiusmonic}g_v=\lambda_{\textup{crit},v}(f)\end{equation} by \cite[Lemmas 2.1 and 2.2]{Ingram:PCF}. We retain the two different notations for concepts that coincide over all relevant places in order to make it clear which concept (splitting radius or critical height) is the applicable one.
	
	We will also be using a measure of the size of a set of bad places for a given polynomial $f(z)\in K[z]$.
	
	\begin{definition*} For $0<\delta<1$, $f\in K[z]$ and $S\subseteq M_K^0$, a \emph{$\delta$-slice of places} $v\in S$ is a set $S'$ of bad places $v\in S$ of $f$ such that \[\sum_{v\in S'}r_v\lambda_{\textup{crit},v}(f)\ge\delta\sum_{v\in S}r_v\lambda_{\textup{crit},v}(f).\] We remark that $S$ and $S'$ may be infinite. \end{definition*}

	Finally, we introduce the potential-theoretic terminology that will be used at places of bad reduction. Let $v\in M_K^0$, let $E\subseteq\mathbf{A}_v^1$, and let $\nu$ be a probability measure with support contained in $E$. The potential function of $\nu$ is by definition \[p_\nu(z)=\int_E-\log\delta_v(z,w)d\nu(w),\] and the energy integral of $\nu$ is \[I(\nu)=\int_E p_\nu(z)d\nu(z).\] The integrals here are Lebesgue integrals; the function $\delta_v(z,w)$ is upper semicontinuous \cite[Proposition 4.1(A)]{BakerRumely}, so $-\log\delta_v(z,w)$ is lower semicontinuous, and hence Borel measurable relative to the $\sigma$-algebra generated by the Berkovich topology.

	The capacity of $E$ is \begin{equation}\label{eqn:gammadef}\gamma(E):=e^{-\inf_\nu I(\nu)}.\end{equation} If $E$ is compact and $\gamma(E)>0$, there is a unique probability measure $\mu_E$ on $E$ for which $I(\mu_E)=\inf_\nu I(\nu)$ \cite[Proposition 7.21]{BakerRumely}. This measure $\mu_E$ is called the \emph{equilibrium measure} for $E$. In the case where $E$ is a disk, $\mu_E$ is simply a Dirac mass at the unique boundary point of $E$, and hence $\gamma(E)$ equals the radius of $E$, a pair of facts we will use heavily. When $E$ is compact, the capacity coincides with the quantity \[\lim_{n\to\infty}\sup\left\{\prod_{i\ne j}\delta_v(z_i,z_j)^{1/(n(n-1))}:z_1,\dots,z_n\in E\right\},\] which is known as the transfinite diameter of $E$ \cite[Theorem 6.24]{BakerRumely}. For a set $T\subseteq\mathbf{A}_v^1$ of $n$ points $z_1,\dots,z_n$, let \begin{equation}\label{eqn:dvdef}d_v(T):=\prod_{i\ne j}\delta_v(z_i,z_j)^{1/(n(n-1))}.\end{equation}
	
	As the non-archimedean analogue of Frostman's Theorem plays an important role in \S\ref{section:equidistribution}, we state it here. 
	
	\begin{thm}\cite[Theorem 6.18]{BakerRumely}\label{thm:Frostman} Let $E\subseteq\mathbf{A}_v^1$ be a compact set of positive capacity, with equilibrium measure $\mu$. Then $p_\mu(z)=I(\mu)$ for all $z\in E$ outside of a set of capacity $0$. \end{thm} In particular, when $E$ is a finite union of disjoint disks, so that $\mu$ is an atomic measure supported exactly on the unique boundary points of these disks, one has $p_\mu(z)=I(\mu)$ for all $z\in E$. Moreover, we note for future use that when $E=\bigsqcup\mathcal{E}_i$ is a union of sets such that $\mathcal{E}_i\cap\mathcal{E}_j\ne\emptyset\Longrightarrow\mathcal{E}_i=\mathcal{E}_j$, and $\nu_E:=\sum k_i\mu_{\mathcal{E}_i}$ where $\mu_{\mathcal{E}_i}$ is the equilibrium measure on $\mathcal{E}_i$ and the $k_i$ are non-negative real numbers with $\sum k_i=1$, Theorem \ref{thm:Frostman} implies that \begin{equation}\label{eqn:prepoteq}p_{\nu_E}(z)=-\sum_{\mathcal{E}_j=\mathcal{E}_i}k_j\log\gamma(\mathcal{E}_i)-\sum_{\mathcal{E}_j\ne\mathcal{E}_i} k_j\log\delta_v(\mathcal{E}_j,\mathcal{E}_i)\hspace{10mm}\forall z\in\mathcal{E}_i\end{equation} and thus \begin{equation}\label{eqn:poteq}I(\nu_E)=\sum_i k_i\left(-\sum_{\mathcal{E}_j=\mathcal{E}_i}k_j\log\gamma(\mathcal{E}_i)-\sum_{\mathcal{E}_j\ne\mathcal{E}_i} k_j\log\delta_v(\mathcal{E}_j,\mathcal{E}_i)\right).\end{equation}
	
	\section{Global quantitative equidistribution}\label{section:equidistribution}
	
	The main result of this section is Theorem \ref{thm:globaleq2}, which is a key precursor to our uniform global quantitative equidistribution theorem over degree $d$ polynomials (Theorem \ref{thm:globaleq}). The argument shares a similar basic idea as an analogous precursor to the main equidistribution theorem of \cite[Corollary 3.8]{Looper:UBCunicrit}: the primary difference is that one must account for the many possible large scale structures of the filled Julia set at a place $v\in M_K^0\setminus\mathscr{S}_d$ of bad reduction, in contrast with the unicritical case, where only one such structure occurs. Below we introduce the key concepts of \emph{wing decompositions} and \emph{$\epsilon$-equidistribution}, before introducing Theorem \ref{thm:globaleq2} and outlining the main idea behind its local inputs, Propositions \ref{prop:capacityconverge} and \ref{prop:uniformeq}.
	
	Let $v\in M_K^0\setminus\mathscr{S}_d$, and let $f(z)\in\CC_v[z]$ of degree $d\ge2$ have bad reduction. Let $g_v$ be the splitting radius of $f$, let $\mathcal{E}$ be the unique disk of radius $\textup{exp}(g_v)$ containing $\mathcal{K}_v$, and let \begin{equation}\label{eqn:Emdef} \mathcal{E}_m:=f^{-m}(\mathcal{E}).\end{equation} Let $\mathcal{B}_{1,1},\dots,\mathcal{B}_{1,d}$ be the disk components of $\mathcal{E}_1$, listed with multiplicity. Similarly, for $m\ge2$, list the disk components $\{\mathcal{B}_{m,i}\}_{i=1}^{d^m}$ of $\mathcal{E}_m$ inductively, so that $\mathcal{B}_{m,i}\subseteq\mathcal{B}_{1,j}$ for $\lceil \frac{i}{d^{m-1}}\rceil=j$, and according to multiplicity. 
	
	A \emph{wing decomposition} of $\mathcal{E}_1$ is a partition of $\mathcal{E}_1$ into two nonempty disjoint sets (wings) $A$ and $B$ with the following properties:
	\begin{itemize}\item $A$ and $B$ are unions of disk components of $\mathcal{E}_1$ \item For any disk components $\mathcal{B}_{1,i}$ of $A$ and $\mathcal{B}_{1,j}$ of $B$, we have  $\log\delta_v(\mathcal{B}_{1,i},\mathcal{B}_{1,j})=g_v$. \end{itemize} A wing decomposition is not unique in general. We note that wing decompositions always exist, since the smallest disk containing $\mathcal{K}_v$ is $\mathcal{E}$, and $\log\text{diam}(\mathcal{E})=g_v$.
	
	\begin{definition*} Let $\epsilon>0$. We say that a finite set $T\subseteq\CC_v$ (not necessarily contained in $\mathcal{E}_1$) is $\epsilon$-equidistributed (at the place $v$) if for every wing decomposition of $\mathcal{E}_1$ as above, we have \begin{equation}\label{eqn:epsilonequidistribution} |T\cap A|>\left(\frac{1-\epsilon}{d}\right)|T|.\end{equation}\end{definition*} The main stepping stone to Theorem \ref{thm:globaleq} is the following theorem.
	
	\begin{thm}\label{thm:globaleq2}Let $\epsilon>0$, let $0<\delta<1$, and let $d\ge2$. There are constants $N$, $M$, $\xi>0$, and $\kappa>0$ depending only on $\epsilon$, $\delta$, and $d$, with the following property. Let $f(z)\in K[z]$ be a polynomial of degree $d$, and let $T\subseteq K$ be a finite set. If $|T|\ge N$, $h_{\textup{crit}}(f)\ge M$, \[\sum_{v\in M_K^0\setminus\mathscr{S}_d}r_v\lambda_{\textup{crit},v}(f)\ge(1-\xi)h_{\textup{crit}}(f),\] and \[\frac{1}{|T|}\sum_{P_i\in T}\hat{h}_f(P_i)\le\kappa h_{\textup{crit}}(f),\] then $T$ is $\epsilon$-equidistributed for a $\delta$-slice of bad places $v\in M_K^0\setminus\mathscr{S}_d$.\end{thm}
	
	Let \begin{equation}\label{eqn:k}\vec{k}=(k_{1},\dots,k_{d})\in(\mathbb{Q}\cap[0,1])^{d}\end{equation} with $\sum_{i=1}^{d}k_{i}=1$. For $v\in M_K^0$ a place of bad reduction for $f\in\CC_v[z]$ of degree $d\ge2$ with filled Julia set $\mathcal{K}_v$, we say a nonempty finite set $T\subseteq\CC_v$ is \emph{$\vec{k}$-distributed} (with respect to $f$ and $v$) if for any $1\le i\le d$, \[ \sum_{\mathcal{B}_{1,j}=\mathcal{B}_{1,i}}k_j|T|=|T\cap\mathcal{B}_{1,i}|,\] where the $\mathcal{B}_{1,i}$ are listed with multiplicity. Here the $k_i$ are defined with respect to the indexing of the $\mathcal{B}_{1,i}$ that we have fixed from the outset. Note that when some $\mathcal{B}_{1,i}$ has multiplicity greater than $1$, $\vec{k}$ is not in general uniquely determined by $T$. We will typically omit the mention of $f$ and $v$ from here on out, as the implied dependence is clear.
	
	In proving the potential-theoretic propositions needed in the proof of Theorem \ref{thm:globaleq}, we will assume $f(z)\in\CC_v[z]$ is a \emph{monic} polynomial having bad reduction, and splitting radius $g_v$. This hypothesis is mathematically inessential, and merely serves to simplify the presentation. 
	
	Let us outline the idea behind the proofs of Propositions \ref{prop:capacityconverge} and \ref{prop:uniformeq}, which are key in proving the local result (Proposition \ref{prop:wtdcap}) underlying Theorem \ref{thm:globaleq}. The goal is to consider modifications of the usual capacity of the Julia set, which assign varying weights to the subsets of $\mathcal{K}_v$ contained in different disk components $\mathcal{E}_{1,i}$ of $\mathcal{E}_1$. A key observation is that once a large scale structure of the filled Julia set $\mathcal{K}_v$ is fixed, and weights are given on each disk component in this structure, the energy corresponding to this set of weights is determined. We explain this idea further after Equation (\ref{eqn:limcaps}). From there, one considers only sets of weights corresponding to a failure of $\epsilon$-equidistribution. Pairs of what we will call balanced $1$-structures and weight vectors can be given a topology such that the subset in question is compact. It is the compactness of this space that allows us to reduce a question about an infinite collection of possible structures and weights to standard facts that hold for any single filled Julia set. 
	
	In order to make these ideas precise, we introduce the following definition.

	\begin{definition*} Let $v\in M_K^0$, and let $f\in\CC_v[z]$ be monic of degree $d\ge2$, with splitting radius $g_v>0$. Let $m_0\ge1$. An \emph{$m_0$-structure} (with respect to $f$ and $v$) is an element \begin{equation}\label{eqn:tuple} (r_{1,1},\dots,r_{1,d^{m_0}},\dots,r_{d^{m_0},1},\dots,r_{d^{m_0},d^{m_0}})\in(-\infty,1]^{d^{2m_0}}\end{equation} such that there is a union of closed disks \begin{equation}\label{eqn:Fm0}\mathcal{F}_{m_0}=\bigcup_{i=1}^{d^{m_0}}\mathfrak{b}_{m_0,i}\subseteq\mathbf{A}_v^1,\end{equation} with no $\mathfrak{b}_{m_0,i}$ properly contained in any $\mathfrak{b}_{m_0,j}$, satisfying \begin{equation}\label{eqn:radii} \frac{\log\delta_v(\mathfrak{b}_{m_0,i},\mathfrak{b}_{m_0,j})}{g_v}=r_{i,j}\end{equation} for all $1\le i,j\le d^{m_0}$, \begin{equation}\label{eqn:admissible}\max_{i,j}\{r_{i,j}\}=1,\end{equation} and \begin{equation}\label{eqn:cap} \frac{\log\gamma(\mathcal{F}_{m_0})}{g_v}=\fr{1}{d^{m_0}}.\end{equation} (We remind the reader that $\gamma$ was defined in (\ref{eqn:gammadef}).) If, additionally, \begin{equation}\label{eqn:smallerinterval}(r_{1,1},\dots,r_{1,d^{m_0}},\dots,r_{d^{m_0},1},\dots,r_{d^{m_0},d^{m_0}})\in[-d^{m_0},1]^{d^{2m_0}},\end{equation} and \begin{equation}\label{eqn:meshcondition}\mu(\mathfrak{b}_{m_0,i})=\frac{d_i}{d^{m_0}},\end{equation} where $\mu$ is the equilibrium measure on $\mathcal{F}_{m_0}$, and $d_i$ is the number of indices $1\le j\le d^{m_0}$ for which $\mathfrak{b}_{m_0,i}=\mathfrak{b}_{m_0,j}$, then we call (\ref{eqn:tuple}) a \emph{balanced $m_0$-structure}.\end{definition*} We refer to such an $\mathcal{F}_{m_0}$ as an \emph{underlying set} of the $m_0$-structure $\Sigma$. We note for clarity that a key instance of an $\mathcal{F}_{m_0}$ is given by $\mathcal{E}_{m_0}$, with the $\mathfrak{b}_{m_0,i}$ its disk components. The purpose of this construction is to situate each $\mathcal{E}_{m_0}$ in a larger ``filled-in'' space of sets that have a similar nature, so that our topological argument can be used. 
	
	Given $\Sigma$ a \textbf{balanced} $m_0$-structure, let a \emph{$\Sigma$-mesh} be a sequence $\{\Sigma_m\}_{m=m_0}^\infty$ of (not necessarily balanced for $m\ne m_0$) $m$-structures having underlying sets \begin{equation}\label{eqn:diskdecomp} \mathcal{F}_m=\bigcup_{i=1}^{d^m}\mathfrak{b}_{m,i}\end{equation} that satisfy \[\mathfrak{b}_{m,i}\subseteq\mathfrak{b}_{m-1,j}\] for $j=\lceil\fr{i}{d}\rceil$ for all $m>m_0$, and \begin{equation}\label{eqn:levelmcap} \mu_m(\mathfrak{b}_{m_0,i}\cap\mathcal{F}_m)=\frac{d_i}{d^{m_0}},\end{equation} where $\mu_m$ is the equilibrium measure on $\mathcal{F}_m$, and for each $1\le i\le d^{m_0}$, $d_i$ is the number of indices $1\le j\le d^{m_0}$ for which $\mathfrak{b}_{m_0,i}=\mathfrak{b}_{m_0,j}$. We note that given a balanced $m_0$-structure, such a sequence must exist, which we now briefly prove by induction. Suppose $m\ge m_0$, and that $\Sigma$ is an $m$-structure with underlying set $\mathcal{F}_m$. First choose any disk component $\mathfrak{b}_{m_0,i}$ of $\mathcal{F}_{m_0}$. Let the disk components of $\mathfrak{b}_{m_0,i}\cap\mathcal{F}_m$ be indexed by $\mathcal{J}$. In each disk component $\mathfrak{b}_{m,j}$ of $\mathcal{F}_m\cap\mathfrak{b}_{m_0,i}$, we put $d$ disks $\mathfrak{d}_{j,1},\dots,\mathfrak{d}_{j,d}$ with centers $c_{j,1},\dots,c_{j,d}\in\mathfrak{b}_{m,j}\cap\CC_v$ having pairwise distance $\text{diam}(\mathfrak{b}_{m,j})$ from each other, and diameter equal to $\text{diam}(\mathfrak{b}_{m,j})$ (so that the disks in fact equal $\mathfrak{b}_{m,j}$ itself). We then shrink the $\mathfrak{d}_{j,k}$ (for all $j\in\mathcal{J}$ and all $1\le k\le d$) about their centers $c_{j,k}$ until the following condition holds. Let $\nu$ be the probability measure \[\nu=\frac{d_{i}}{d^{m_0}}\mu_\mathfrak{d}+\frac{d^{m_0}-d_{i}}{d^{m_0}}\mu_\mathfrak{b},\] where $\mu_\mathfrak{d}$ is the equilibrium measure on \[\mathcal{D}:=\bigcup_{\substack{j\in\mathcal{J}\\1\le k\le d}}\mathfrak{d}_{j,k},\] and $\mu_{\mathfrak{b}}$ is the equilibrium measure on $\bigcup_{\mathfrak{b}_{m_0,l}\ne\mathfrak{b}_{m_0,i}}\mathfrak{b}_{m_0,l}\cap\mathcal{F}_m$. We shrink the radii of the disks $\mathfrak{d}_{j,k}$ (uniformly, say) until, for each $x\in\mathcal{D}$, \[\int\log\delta_v(x,y)d\nu(y)=\frac{g_v}{d^{m+1}}.\] The disks so obtained are pairwise disjoint, by our condition on the centers. Repeating an analogous procedure for each disk component $\mathfrak{b}_{m_0,i}$ of $\mathcal{F}_{m_0}$, and taking the union of the resulting shrunk disks, we obtain an $\mathcal{F}_{m+1}$ that is an underlying set of an $(m+1)$-structure, along with equilibrium measure $\nu$ satisfying $\nu(\mathfrak{b}_{m_0,i}\cap\mathcal{F}_{m+1})=\frac{d_i}{d^{m_0}}$ for all $i$ (corresponding to our target condition (\ref{eqn:levelmcap})). Indeed, by Frostman's Theorem (Theorem \ref{thm:Frostman}), the equilibrium measure $\eta$ on a union $\mathcal{C}$ of disks is characterized by the property that the potentials are everywhere equal, i.e., that the function $\psi:\mathcal{C}\to\mathbb{R}$ given by \[\psi(x)=\int\log\delta_v(x,y)d\eta(y)\] is constant. Our construction of $\nu$ thus ensures that $\nu$ is in fact the equilibrium measure on $\mathcal{F}_{m+1}$, and that (\ref{eqn:levelmcap}) holds, with $m+1$ replacing $m$. (We remark that a more careful argument allows us to construct a mesh of \emph{balanced} $m$-structures, but this is unnecessary for our purposes.)
	
	Fix $\vec{k}$ as in (\ref{eqn:k}). For a balanced $1$-structure $\Sigma$ with $\Sigma$-mesh $\{\Sigma_m\}$ and underlying sets $\{\mathcal{F}_m\}$, define for each $m\ge 1$ the unique probability measure $\mu_{\vec{k},m}=\mu_{\vec{k},m}(\Sigma)$ on $\mathcal{F}_m$ such that for all $1\le i\le d$, \[\mu_{\vec{k},m}(\mathfrak{b}_{1,i}\cap\mathcal{F}_m)=\sum_{\mathfrak{b}_{1,j}=\mathfrak{b}_{1,i}}k_j,\] and $\mu_{\vec{k},m}$ is a scalar multiple of the equilibrium measure on $\mathfrak{b}_{1,i}\cap\mathcal{F}_m$. Let \begin{equation}\label{eqn:limcaps}\mathcal{I}_{\vec{k}}(\Sigma)=\lim_{m\to\infty}I(\mu_{\vec{k},m}(\Sigma)).\end{equation} Note that the sequence defining this limit is increasing and bounded above (as may be seen in several ways, including immediately via Proposition \ref{prop:weld1}). Thus this limit exists by the monotone convergence theorem. We also highlight the following important observation:
	
	\begin{keyfact}$\mathcal{I}_{\vec{k}}(\Sigma)/g_v$ depends only on $d$, $\vec{k}$ and on $\Sigma$. In particular, it is independent of the choice of mesh $\{\Sigma_m\}_{m=1}^\infty$ (other than the fact that it clearly depends on $\Sigma=\Sigma_1$) and of the underlying sets $\mathcal{F}_m$, and it is also independent of the $v$ and $f$ giving rise to $\Sigma$. \end{keyfact} \noindent 
	
	This is because for any $m\ge 1$, Condition (\ref{eqn:cap}) gives \[\frac{\log\gamma(\mathcal{F}_m)}{g_v}=\frac{1}{d^m}\] for all $m\ge1$, and so by (\ref{eqn:levelmcap}) applied in the case $m_0=1$ and Frostman's Theorem (Theorem \ref{thm:Frostman}), specifying the vector of radii given by the balanced $1$-structure determines each of the \begin{equation*}\tau_i:=\frac{\log\gamma(\mathfrak{b}_{1,i}\cap\mathcal{F}_m)}{g_v}\end{equation*} for $1\le i\le d$, the full set of which in turn determines $I(\mu_{\vec{k},m}(\Sigma))/g_v$ once $\vec{k}$ is specified. Indeed, letting $\mu_m$ once again be the equilibrium measure on $\mathcal{F}_m$, Frostman's Theorem yields \[\sum_{\mathfrak{b}_{1,j}=\mathfrak{b}_{1,1}}k_j\tau_j+\frac{1}{g_v}\sum_{\mathfrak{b}_{1,j}\ne\mathfrak{b}_{1,1}}k_j\log\delta_v(\mathfrak{b}_{1,1}\cap\mathcal{F}_m,\mathfrak{b}_{1,j})=-I(\mu_m)/g_v\] \[\vdots\] \[\sum_{\mathfrak{b}_{1,j}=\mathfrak{b}_{1,d}}k_j\tau_j+\frac{1}{g_v}\sum_{\mathfrak{b}_{1,j}\ne\mathfrak{b}_{1,d}}k_j\log\delta_v(\mathfrak{b}_{1,d}\cap\mathcal{F}_m,\mathfrak{b}_{1,j})=-I(\mu_m)/g_v,\] a system of $d$ linear equations in the $d$ unknowns $\tau_i$, which has a unique solution by the uniqueness of the equilibrium measure on $\mathcal{F}_m$. For $\vec{\tau}=(\tau_1,\dots,\tau_d)$ this unique solution, (\ref{eqn:poteq}) then gives \[-I(\mu_{\vec{k},m}(\Sigma))/g_v=\sum_{i=1}^dk_i\left(\sum_{\mathfrak{b}_{1,j}=\mathfrak{b}_{1,i}}k_j\tau_j+\sum_{\mathfrak{b}_{1,j}\ne\mathfrak{b}_{1,i}}k_j\log\delta_v(\mathfrak{b}_{1,i}\cap\mathcal{F}_m,\mathfrak{b}_{1,j})/g_v \right).\] Hence $\{I(\mu_{\vec{k},m}(\Sigma))/g_v\}$ and $\mathcal{I}_{\vec{k}}(\Sigma)/g_v$ are independent of the mesh $\{\Sigma_m\}$ other than at the index $m=1$. We use this fact to view the $I(\mu_{\vec{k},m}(\Sigma))/g_v$ and $\mathcal{I}_{\vec{k}}(\Sigma)/g_v$ as functions of $\Sigma$ only, and not a particular choice of $\Sigma$-mesh.
	\newline
	
	Let $\mu_{\vec{k}}=\mu_{\vec{k}}(\Sigma)$ be any weak$^*$ subsequential limit of the $\mu_{\vec{k},m}(\Sigma)$. We note that \begin{equation}\label{eqn:weaklimit}\mathcal{I}_{\vec{k}}(\Sigma)=I(\mu_{\vec{k}}(\Sigma)),\end{equation} which is proved exactly as the Claim in the proof of \cite[Proposition 3.7]{Looper:UBCunicrit}. Write \begin{equation}\label{eqn:gammamuk}\gamma(\mu_{\vec{k}})=\textup{exp}(-I(\mu_{\vec{k}}(\Sigma)))=\textup{exp}(-\mathcal{I}_{\vec{k}}(\Sigma)).\end{equation} We introduce a proposition giving a uniform rate of convergence for the limit (\ref{eqn:limcaps}). It is key in the proof of Proposition \ref{prop:capacityconverge}.
	
	\begin{prop}\label{prop:weld1} Let $f(z)\in\CC_v[z]$ be monic of degree $d\ge2$ having bad reduction at $v\in M_K^0\setminus\mathscr{S}_d$, and let $\Sigma$ be a balanced $1$-structure with respect with to $f$ and $v$. Let $\vec{k}$ be as in (\ref{eqn:k}). Then for all $m\ge1$, \[-I(\mu_{\vec{k},m}(\Sigma))+I(\mu_{\vec{k},m+1}(\Sigma))\le \left(\frac{1}{d^{m-1}}-\frac{1}{d^m}\right)g_v.\]\end{prop} 
	
	\begin{proof}[Proof of Proposition \ref{prop:weld1}] Let $\{\mathcal{F}_m\}$ be a sequence of underlying sets associated to a $\Sigma$-mesh. Write $\vec{k}=(k_1,\dots,k_d)$ and \[\mathcal{F}_1=\bigcup_{i=1}^d\mathfrak{b}_{1,i},\] where whenever $i\ne j$, either $\mathfrak{b}_{1,i}=\mathfrak{b}_{1,j}$ or $\mathfrak{b}_{1,i}\cap\mathfrak{b}_{1,j}=\emptyset$. For each $1\le i\le d$, let $d_i$ be the number of indices $1\le j\le d$ for which $\mathfrak{b}_{1,i}=\mathfrak{b}_{1,j}$. For each $m\ge 1$, the following inequality results from (\ref{eqn:poteq}) (where the last few lines are explained afterward): 
		\begin{align*}-I(\mu_{\vec{k},m}(\Sigma))+I(\mu_{\vec{k},m+1}(\Sigma))=&\sum_{i=1}^d k_i\left(\sum_{\mathfrak{b}_{1,j}=\mathfrak{b}_{1,i}}k_j\log\gamma(\mathfrak{b}_{1,i}\cap\mathcal{F}_m)+\sum_{\mathfrak{b}_{1,j}\ne \mathfrak{b}_{1,i}}k_j \log\delta_v(\mathfrak{b}_{1,i},\mathfrak{b}_{1,j})\right)\\&-\sum_{i=1}^d k_i\left(\sum_{\mathfrak{b}_{1,j}=\mathfrak{b}_{1,i}}k_j\log\gamma(\mathfrak{b}_{1,i}\cap\mathcal{F}_{m+1})+\sum_{\mathfrak{b}_{1,j}\ne \mathfrak{b}_{1,i}}k_j \log\delta_v(\mathfrak{b}_{1,i},\mathfrak{b}_{1,j})\right)\\=&\sum_{i=1}^dk_i\left(\sum_{\mathfrak{b}_{1,j}=\mathfrak{b}_{1,i}}k_j\left(\log\gamma(\mathfrak{b}_{1,i}\cap\mathcal{F}_m)-\log\gamma(\mathfrak{b}_{1,i}\cap\mathcal{F}_{m+1})\right)\right) \\ \le&\sum_{i=1}^dk_i\left(\sum_{\mathfrak{b}_{1,j}=\mathfrak{b}_{1,i}}k_j\max_{1\le l\le d}\{\log\gamma(\mathfrak{b}_{1,l}\cap\mathcal{F}_m)-\log\gamma(\mathfrak{b}_{1,l}\cap\mathcal{F}_{m+1})\}\right) \\\le&\left(\sum_{i=1}^dk_i\right)^2\max_{1\le l\le d}\{\log\gamma(\mathfrak{b}_{1,l}\cap\mathcal{F}_m)-\log\gamma(\mathfrak{b}_{1,l}\cap\mathcal{F}_{m+1})\}\\=&\max_{1\le l\le d}\{\log\gamma(\mathfrak{b}_{1,l}\cap\mathcal{F}_m)-\log\gamma(\mathfrak{b}_{1,l}\cap\mathcal{F}_{m+1})\}\\\le&d\max_{1\le l\le d}\left\{\frac{d_l}{d}\biggl(\log\gamma(\mathfrak{b}_{1,l}\cap\mathcal{F}_m)-\log\gamma(\mathfrak{b}_{1,l}\cap\mathcal{F}_{m+1})\biggr)\right\}\\=&d\max_{1\le l\le d}\biggl\{\frac{d_l}{d}\log\gamma(\mathfrak{b}_{1,l}\cap\mathcal{F}_m)+\sum_{\mathfrak{b}_{1,j}\ne \mathfrak{b}_{1,l}}\frac{d_j}{d}\log\delta_v(\mathfrak{b}_{1,l},\mathfrak{b}_{1,j})\\&-\frac{d_l}{d}\log\gamma(\mathfrak{b}_{1,l}\cap\mathcal{F}_{m+1})-\sum_{\mathfrak{b}_{1,j}\ne \mathfrak{b}_{1,l}}\frac{d_j}{d}\log\delta_v(\mathfrak{b}_{1,l},\mathfrak{b}_{1,j})\biggr\}\\=&d\left(\log\gamma(\mathcal{F}_m)-\log\gamma(\mathcal{F}_{m+1})\right) \\=& d\left(\frac{1}{d^m}-\frac{1}{d^{m+1}}\right)g_v. \end{align*} Here, the second to last equality holds since for all $m\ge 1$, Frostman's Theorem (Theorem \ref{thm:Frostman}) along with (\ref{eqn:levelmcap}) implies that \[\frac{d_i}{d}\log\gamma(\mathfrak{b}_{1,i}\cap\mathcal{F}_m)+\sum_{\mathfrak{b}_{1,j}\ne \mathfrak{b}_{1,i}}\frac{d_j}{d}\log\delta_v(\mathfrak{b}_{1,i},\mathfrak{b}_{1,j})=\log\gamma(\mathcal{F}_m)\] for all $1\le i\le d$ (and similarly for $\mathcal{F}_{m+1}$), and the last equality follows from (\ref{eqn:cap}).\end{proof}
	
	For a nonempty finite set $T\subseteq\mathcal{E}_m$, let $\mu_{m,T}$ be the unique probability measure on $\mathcal{E}_m$ such that for each $i$, the measure $\mu_{m,T}$ restricts to a scalar multiple of the equilibrium measure on $\mathcal{B}_{m,i}$, and \begin{equation}\label{eqn:mumt}\mu_{m,T}(\mathcal{B}_{m,i})=\frac{|T\cap\mathcal{B}_{m,i}|}{|T|}.\end{equation} We use these measures, along with the following lemma, in the proof of Proposition \ref{prop:capacityconverge}. We remind the reader that the $\mathcal{B}_{m,i}$ are defined as the disk components of $\mathcal{E}_m:=f^{-m}(\mathcal{E})$, where $\mathcal{E}$ is the smallest disk containing $\mathcal{K}_v$. The following lemma will also be used in the proof of Proposition \ref{prop:capacityconverge}.
	
	\begin{lem}\label{lem:radiuslb}Let $v\in M_K^0\setminus\mathscr{S}_d$, let $f(z)\in\CC_v[z]$ be a monic polynomial of degree $d\ge2$ having bad reduction at $v$, and let $g_v$ be the splitting radius of $f$. Then for all $m\ge1$, \[\log\text{diam}(\mathcal{B}_{m,i})\ge-d^mg_v.\]\end{lem}
	
	\begin{proof}The proof is parallel to that of \cite[Lemma 3.3]{Looper:UBCunicrit}. One may also prove this more simply by combining Frostman's Theorem (Theorem \ref{thm:Frostman}) with the fact that $\mu(\mathcal{B}_{m,i})=\frac{d_i}{d^m}$ for $\mu$ the equilibrium measure on $\mathcal{E}_m$ and $d_i$ an integer between $1$ and $d^m-1$.\end{proof} In particular, given $f(z)\in\CC_v[z]$ a monic polynomial of degree $d\ge2$ having bad reduction at $v\in M_K^0\setminus\mathscr{S}_d$ and splitting radius $g_v$, it follows from our definition of balanced $1$-structures and $\Sigma$-meshes, Lemma \ref{lem:radiuslb}, and the fact \cite[Theorem 1.2]{DeMarcoRumely} that for all $m\ge1$, \begin{equation}\label{eqn:capEm}\log\gamma(\mathcal{E}_m)=\frac{1}{d^m}g_v,\end{equation} that there is a balanced $1$-structure $\Sigma_1$ and mesh $\{\Sigma_m\}_{m=1}^\infty$ with underlying sets consisting of $\mathcal{E}_m$ for all $m\ge1$. 
	
	\begin{prop}\label{prop:capacityconverge} Let $\epsilon>0$, and let $d\ge2$. There exist constants $M=M(d,\epsilon)$ and $N=N(M,d,\epsilon)$ with the following property. Let $v\in M_K^0\setminus\mathscr{S}_d$, let $f(z)\in\CC_v[z]$ be a monic polynomial of degree $d\ge2$ having bad reduction at $v$ and splitting radius $g_v$, and let $\vec{k}$ be as in (\ref{eqn:k}). Let $\gamma(\mu_{\vec{k}})$ as in (\ref{eqn:gammamuk}) be associated to the natural mesh $\{\Sigma_m\}_{m=1}^\infty$ given by the underying sets $\{\mathcal{E}_m\}_{m=1}^\infty$. For each $m\ge M$, every $\vec{k}$-distributed set $T\subseteq\mathcal{E}_m\cap\CC_v$ of order $n\ge N$ satisfies \[d_v(T)\le\gamma(\mu_{\vec{k}})e^{\epsilon g_v},\] where $d_v$ is defined as in (\ref{eqn:dvdef}).\end{prop}
	
	\begin{proof} Let $m\ge1$, and let $\vec{k}$ be as in (\ref{eqn:k}). The key to the proof is the following claim.
		
		\textbf{Claim}: There is an $N'=N'(d,\epsilon,m)$ such that if $T\subseteq\mathcal{E}_m\cap\CC_v$ is a set of $n\ge N'$ elements $z_1,\dots,z_n\in\mathcal{E}_m$, then \begin{equation}\label{eqn:weld2}d_v(T)\le e^{(\epsilon/2)g_v}\textup{exp}(-I(\mu_{m,T})),\end{equation} where $\mu_{m,T}$ is as in (\ref{eqn:mumt}). 
		
		Proof of claim: Write $j_{m,i}=\mu_{m,T}(\mathcal{B}_{m,i})$ and \[\mathcal{E}_m=\bigcup_{i=1}^{s_m}\mathcal{B}_{m,i}\] for disjoint disks $\mathcal{B}_{m,i}$, and let \[r_m=\min_{1\le i\le s_m}\{\textup{diam}(\mathcal{B}_{m,i})\}.\] For any $w\in T\cap\mathcal{B}_{m,i}$, we have 
		\begin{equation}\label{ineq:diskdiam}\begin{split}\textup{exp}(-p_{\mu_{m,T}}(w))&\ge\left(\textup{diam}(\mathcal{B}_{m,i})\prod_{\substack{z\in T\\z\ne w}}\delta_v(z,w)\right)^{1/n}\\&\ge\left(r_m\prod_{\substack{z\in T\\z\ne w}}\delta_v(z,w)\right)^{1/n},\end{split}\end{equation} since the first inequality is an equality in the case maximizing the right-hand side, i.e., whenever for any $1\le i\le s_m$ and any distinct $z,w\in T\cap\mathcal{B}_{m,i}$, one has $\delta_v(z,w)=\textup{diam}(\mathcal{B}_{m,i})$. We rewrite (\ref{ineq:diskdiam}) as \begin{equation}\label{eqn:individualpotential} \left(\frac{1}{r_m}\right)^{1/n}\textup{exp}(-p_{\mu_{m,T}}(w))\ge \prod_{\substack{z\in T\\z\ne w}}\delta_v(z,w)^{1/n}.\end{equation} Taking the geometric mean of both sides over all $w\in T$ gives \begin{equation*} \left(\frac{1}{r_m}\right)^{1/n}\textup{exp}(-I(\mu_{m,T}))\ge d_v(T)^{(n-1)/n},\end{equation*} and hence, raising both sides to the $\left(n/(n-1)\right)$-th power, \begin{equation*} \left(\left(\frac{1}{r_m}\right)^{1/(n-1)}\textup{exp}(-I(\mu_{m,T}))^{1/(n-1)}\right)\textup{exp}(-I(\mu_{m,T}))\ge d_v(T).\end{equation*} As $\log r_m\ge-d^mg_v$ by Lemma \ref{lem:radiuslb}, and trivially $-I(\mu_{m,T})\le g_v$, taking $n\gg_{m,\epsilon}1$ proves the Claim (\ref{eqn:weld2}). On the other hand, Proposition \ref{prop:weld1} applied to our mesh $\{\Sigma_m\}_{m=1}^\infty$ and to our $\vec{k}$ implies that \begin{equation*}\begin{split}
				-I(\mu_{m,T}) &= -I(\mu_{\vec{k},m})
				= -I(\mu_{\vec{k}}) + \sum_{m'=m}^{\infty}
				\Big( -I(\mu_{\vec{k},m'}) +I(\mu_{\vec{k},m'+1}) \Big)
				\\
				&\leq \log\gamma(\mu_{\vec{k}}) + \frac{1}{d^{m-1}} g_v \\ &\leq \log\gamma(\mu_{\vec{k}}) + \frac{\epsilon}{2} g_v
		\end{split}\end{equation*} for all $m\gg_{\epsilon,d}1$, where the first inequality follows from Proposition \ref{prop:weld1}. Combining this with the Claim (\ref{eqn:weld2}) completes the proof.\end{proof}
	
	\begin{prop}\label{prop:uniformeq} Let $\epsilon>0$, and let $d\ge2$. There are integers $N_0=N_0(d,\epsilon)$, $m=m(d,\epsilon)\ge 1$, and a constant $\epsilon'=\epsilon'(d,\epsilon)>0$ with the following property. Let $v\in M_K^0\setminus\mathscr{S}_d$, let $f(z)\in\CC_v[z]$ be a monic polynomial of degree $d$ having bad reduction at $v$, let $g_v$ be the splitting radius of $f$, and let $\mathcal{E}$ be the smallest disk containing the filled Julia set $\mathcal{K}_v$ of $f$. If $T\subseteq f^{-m}(\mathcal{E})$ is a finite set with $|T|\ge N_0$, and $T$ is not $\epsilon$-equidistributed at $v$, then \[\log d_v(T)\le-\epsilon'g_v.\]
	\end{prop}
	
	\begin{proof} Let $\mathfrak{W}\subseteq[-d,1]^{d^2}$ be the set of balanced $1$-structures with respect to monic polynomials $f\in\CC_v[z]$ of degree $d$ having bad reduction at $v$. Let $\mathfrak{K}\subseteq[0,1]^d$ be the set of $d$-tuples as in (\ref{eqn:k}), except with the rationality condition on the $k_i$ omitted (so that the $k_i$ may be irrational reals). Let $\mathfrak{S}$ be the set of elements $(\Sigma,\vec{k})\in\mathfrak{W}\times\mathfrak{K}$ such that either: \begin{itemize} \item $\vec{k}$ has rational coordinates, and any $\vec{k}$-distributed set of points contained in an underlying set of $\Sigma$ fails to be $\epsilon$-equidistributed, or \item $(\Sigma,\vec{k})$ is a limit point of a sequence $(\Sigma,\vec{k}_i)$ such that the previous condition holds. \end{itemize} Referring back to the definition of $\epsilon$-equidistribution (\ref{eqn:epsilonequidistribution}), we see that $\mathfrak{S}$ is nonempty if and only if $\epsilon\le1$. If $\epsilon>1$, then clearly \textbf{no} set $T\subseteq f^{-m}(\mathcal{E})$ can fail to be $\epsilon$-equidistributed at $v$, so the proposition holds vacuously. Thus, we can assume that $\epsilon\le1$, so that $\mathfrak{S}$ is nonempty. Let \[\psi:\mathfrak{S}\to\mathbb{R}\] be given by \[\psi(\Sigma,\vec{k})=\frac{\mathcal{I}_{\vec{k}}(\Sigma)}{g_v};\] here we note that by the Key Fact given before (\ref{eqn:weaklimit}), $\psi$ is well-defined (and only depends on $d$ and $\epsilon$, not on $v$ or $f$) for all $(\Sigma,\vec{k})\in\mathfrak{S}$ with $\vec{k}\in\QQ^d$, and by the density of the rationals in the reals, the functions $\mathcal{I}_{\vec{k}}(\Sigma)$ defined earlier have a unique continuous extension to pairs $(\Sigma,\vec{k})$ where $\vec{k}$ has arbitrary real coordinates in $[0,1]$. We define $\psi$ on $\mathfrak{S}$ via this continuous extension.%%Note that we aren't claiming sequences of $1$-structures $\Sigma_i$ converge within $\Sigma$, which would require the closedness argued below. So continuity here is argued via ``small perturbations in domain lead to small perturbations in range," rather than a limit argument.%%
		
		\textbf{Claim}: $\mathfrak{S}$ is a compact subset of $\mathfrak{W}\times\mathfrak{K}$. 
		\newline
		
		Proof of claim: We begin with a subclaim: If $\{\Sigma^t\}_{t=1}^\infty$ is a convergent sequence of $d^2$-tuples satisfying (\ref{eqn:tuple})--(\ref{eqn:radii}) as well as (\ref{eqn:meshcondition}), with $m_0=1$, and if, for $\Sigma^t$ corresponding to $\mathcal{F}_1^t=\bigcup\mathfrak{b}_{1,i}^t$, we write \[r_i^t=\log\textup{diam}(\mathfrak{b}_{1,i}^t)\hspace{5mm}\forall i,\] and \[\delta_{ij}^t=\delta_v(\mathfrak{b}_{1,i}^t,\mathfrak{b}_{1,j}^t),\] then \begin{equation}\label{eqn:limitprop}\lim_{t\to\infty}\log\delta_{ij}^t=\lim_{t\to\infty}\max\{r_i^t,r_j^t\}\Longrightarrow\lim_{t\to\infty}r_j^t=\lim_{t\to\infty}r_i^t.\end{equation} We will prove this subclaim once the rest of the proof of Proposition \ref{prop:uniformeq} is finished. Assuming the truth of this statement, we see that any limit $\Sigma'$ of such a sequence $\{\Sigma^t\}_{t=1}^\infty$ whose terms satisfy (\ref{eqn:tuple})--(\ref{eqn:radii}) as well as (\ref{eqn:meshcondition}) must have underlying set of the form $\mathcal{F}_1'=\bigcup\mathfrak{b}_{1,i}'$ as in (\ref{eqn:Fm0})--(\ref{eqn:radii}). (In the absence of (\ref{eqn:limitprop}) holding for this sequence of $\Sigma^t$, we could a priori have some $\mathfrak{b}_{1,i}'$ properly contained in some $\mathfrak{b}_{1,j}'$, which we forbade in (\ref{eqn:Fm0}).) Clearly $\mathcal{F}_1'$ must satisfy (\ref{eqn:meshcondition}) as well.
		
		It follows that the set $\mathfrak{V}$ of $d^2$-tuples as in (\ref{eqn:tuple}) with $m_0=1$ that correspond to sets as in (\ref{eqn:Fm0}) satisfying (\ref{eqn:radii}) and (\ref{eqn:smallerinterval})--(\ref{eqn:meshcondition}) is closed in $[-d,1]^{d^2}$. Moreover, the subset $\mathfrak{Z}\subseteq\mathfrak{V}$ corresponding to underlying sets satisfying (\ref{eqn:admissible}) and (\ref{eqn:cap}) is clearly closed in $\mathfrak{V}$. But $\mathfrak{Z}$ is exactly $\mathfrak{W}$, so we conclude that $\mathfrak{W}$ is closed in the topology inherited from the Euclidean topology on $[-d,1]^{d^2}$. Hence $\mathfrak{W}\times\mathfrak{K}$ is compact.
		
		To finish proving the claim, it suffices to show that for any $(\Sigma,\vec{k})\in(\mathfrak{W}\times\mathfrak{K})\setminus\mathfrak{S}$, there is a neighborhood of $(\Sigma,\vec{k})$ in $\mathfrak{W}\times\mathfrak{K}$ that is contained in $(\mathfrak{W}\times\mathfrak{K})\setminus\mathfrak{S}$. To see this, fix $(\Sigma,\vec{k})\in(\mathfrak{W}\times\mathfrak{K})\setminus\mathfrak{S}$, and for any $\epsilon_1,\epsilon_2>0$, let $U_{\epsilon_1}$ denote an $\epsilon_1$-neighborhood of $\Sigma$ in $\mathfrak{W}$, and let $V_{\epsilon_2}$ denote an $\epsilon_2$-neighborhood of $\vec{k}$ in $\mathfrak{K}$. First note that since $\epsilon$-equidistribution is an open condition, there exists an $\epsilon_2>0$ such that $(\Sigma,\vec{k}_{\epsilon_2})\notin\mathfrak{S}$ for all $\vec{k}_{\epsilon_2}\in V_{\epsilon_2}$. We then claim that for any such $\epsilon_2>0$, there is an $\epsilon_1>0$ such that for every $\Sigma_{\epsilon_1}\in U_{\epsilon_1}$ and every $\vec{k}_{\epsilon_2}\in V_{\epsilon_2}$, one has $(\Sigma_{\epsilon_1},\vec{k}_{\epsilon_2})\notin\mathfrak{S}$. Concretely, we may choose any $\epsilon_1<\min_{i,j}\{1-r_{i,j}:r_{i,j}\ne1\}$, where the $r_{i,j}$ are the logarithmic radii as in (\ref{eqn:tuple}) associated to the $1$-structure $\Sigma$ (so that in fact $\epsilon_1$ does not depend on $\epsilon_2$). This proves that $\mathfrak{S}$ is closed in $\mathfrak{W}\times\mathfrak{K}$ (which we showed above is compact), and so we have proved the claim that $\mathfrak{S}$ is compact as well.
		
		Continuing with the proof of Proposition \ref{prop:uniformeq}, we assume that $\epsilon\le1$, so that $\mathfrak{S}$ is nonempty. As $\psi$ is a continuous function on a compact set, it must attain a minimum $\eta=\eta(d,\epsilon)$. Fix $(\Sigma,\vec{k})\in\mathfrak{S}$ be such that $\psi(\Sigma,\vec{k})=\eta$, and let $\{\Sigma_m\}_{m=1}^\infty$ be any mesh for $\Sigma$ with underlying sets $\{\mathcal{F}_m\}_{m=1}^\infty$. As before, write \[\mathcal{F}_1=\bigcup_{i=1}^d\mathfrak{b}_{1,i},\] where $i\ne j$ implies that either $\mathfrak{b}_{1,i}=\mathfrak{b}_{1,j}$ or $\mathfrak{b}_{1,i}\cap\mathfrak{b}_{1,j}=\emptyset$. Let $\mu_{\vec{k}}$ be any weak$^*$ subsequential limit of the measures $\mu_{\vec{k},m}(\Sigma)$. Let $\mathcal{T}=\bigcap_{m=1}^\infty\mathcal{F}_m$. Note that (\ref{eqn:levelmcap}) combined with (\ref{eqn:cap}) and (\ref{eqn:weaklimit}) implies that $0=\mathcal{I}_{\vec{l}}(\Sigma)=I(\mu_{\vec{l}})$ for $\vec{l}=(\frac{1}{d},\dots,\frac{1}{d})$ and $\mu_{\vec{l}}$ any weak$^*$ subsequential limit of the $\mu_{\vec{l},m}(\Sigma)$. Since $\textup{supp}(\mu_{l,m})$ is contained in $\mathcal{F}_m$, we also see that the support of such a $\mu_{\vec{l}}$ is contained in $\mathcal{T}$. Note moreover that (\ref{eqn:cap}) yields $\gamma(\mathcal{T})\le1$. From these observations and the uniqueness of the equilibrium measure, we see that any such $\mu_{\vec{l}}$ is in fact the equilibrium measure on $\mathcal{T}$. By construction, it satisfies \begin{equation}\label{eqn:eqdist}\mu_{\vec{l}}(\mathfrak{b}_{1,i}\cap\mathcal{T})=\frac{d_i}{d}\end{equation} for each $1\le i\le d$, where $d_i$ is the number of indices $1\le j\le d$ for which $\mathfrak{b}_{1,i}=\mathfrak{b}_{1,j}$. On the other hand, for all $m\ge 1$ and all $1\le i\le d$, we have \[\mu_{\vec{k},m}(\mathfrak{b}_{1,i}\cap\mathcal{F}_m)=\sum_{\mathfrak{b}_{1,j}=\mathfrak{b}_{1,i}}k_j,\] which implies that \begin{equation}\label{eqn:kdist}\mu_{\vec{k}}(\mathfrak{b}_{1,i}\cap\mathcal{T})=\sum_{\mathfrak{b}_{1,j}=\mathfrak{b}_{1,i}}k_j.\end{equation} But as $\vec{k}$ came from an ordered pair $(\Sigma,\vec{k})$ corresponding to a failure of $\epsilon$-equidistribution, we also know that there is some $i$ such that $\frac{d_i}{d}\ne\sum_{\mathfrak{b}_{1,j}=\mathfrak{b}_{1,i}}k_j$, meaning that there is some $i$ such that the right-hand sides of (\ref{eqn:eqdist}) and (\ref{eqn:kdist}) are unequal. By the uniqueness of the equilibrium measure on $\mathcal{T}$, it follows that \[0=I(\mu_{\vec{l}})< I(\mu_{\vec{k}})=\mathcal{I}_{\vec{k}}(\Sigma)=g_v\psi(\Sigma,\vec{k})=g_v\eta.\] We conclude that $\eta>0$. Recall that $\eta=\eta(d,\epsilon)$.
		
		Now let $f\in\CC_v[z]$ be any polynomial of degree $d$ having splitting radius $g_v>0$, and let $\mathcal{E}_m=f^{-m}(\mathcal{E})$ for $\mathcal{E}$ the smallest disk containing the filled Julia set $\mathcal{K}_v$ of $f$. We saw in the discussion around (\ref{eqn:capEm}) that there is a balanced $1$-structure $\Sigma_1$ and mesh $\{\Sigma_m\}_{m=1}^\infty$ with underlying sets $\mathcal{E}_m$ such that $\mathcal{K}_v$ is contained in $\mathcal{E}_m$ for all $m\ge1$. Hence by the definition of $\eta$, it follows that for any $\vec{k}$ for which the $\vec{k}$-distributed sets in $\mathcal{E}_1$ fail to be $\epsilon$-equidistributed, and for $\mu_{\vec{k}}$ any weak$^*$ subsequential limit of the $\mu_{\vec{k},m}(\Sigma_1)$, we have \[\gamma(\mu_{\vec{k}})\le e^{-\eta g_v}.\] Letting $\epsilon'=\eta/2$, we then have from Proposition \ref{prop:capacityconverge} that there is an $m=m(d,\epsilon')$ and an $N=N(d,\epsilon')$ such that for any $\vec{k}$-distributed set $T\subseteq\mathcal{E}_m$ of order $\ge N$, \[d_v(T)\le\gamma(\mu_{\vec{k}})e^{\epsilon'g_v}\le e^{-\eta g_v}e^{\epsilon'g_v}=e^{-\eta g_v/2}=e^{-\epsilon'g_v}.\] As noted previously, $\eta$ depends only on $d$ and on $\epsilon$, and hence the same holds for $\epsilon'$. This completes the proof, with the exception of the deferred proof of the subclaim, supplied below. The impatient reader may want to skip to Proposition \ref{prop:wtdcap}.\end{proof}
	
	\noindent\textbf{Proof of Subclaim (\ref{eqn:limitprop})}: Let all $\delta_{ij}^t$ and $r_i^t$ for $1\le i,j\le d$ be as in (\ref{eqn:limitprop}), where the superscripts of $t$ will be indices, not exponents. Define a binary relation on the integers between $1$ and $d$ by \[i\sim j\Longleftrightarrow\lim_{t\to\infty}\log\delta_{ij}^t=\lim_{t\to\infty}r_i^t.\] Note that $\sim$ is reflexive. Let $1\le i\le d$ be such that $\mathcal{X}_i:=\{1\le j\le d:i\sim j\}$ has order at least $2$; we may assume that such an $i$ exists, for otherwise by the reflexivity of $\sim$, Subclaim (\ref{eqn:limitprop}) is clear. By reordering indices if needed we may assume that $i=1$, and that $\mathcal{X}_1=\{1,2,\dots,q\}$. We have by the ultrametric inequality that $\delta^t_{jl}\leq\max\{\delta^t_{1j}, \delta^t_{1l}\}$ for all $1\le j,l\le d$, and hence for all $j\in\mathcal{X}_1$ and all $1\le l\le d$, we have
	
	\begin{equation}\label{eqn:diskinside}\begin{split}\lim_{t\to\infty}\delta_{1l}^t&=\lim_{t\to\infty}\max\{r_1^t,\delta_{1l}^t\}\\&=\lim_{t\to\infty}\max\{\delta_{1j}^t,\delta_{1l}^t\}\\&\ge\lim_{t\to\infty}\delta_{jl}^t,\end{split}\end{equation} where the second equality follows from the fact that $j\in\mathcal{X}_1$. For each $t$, let $\mu^t$ be the equilibrium measure on $\mathcal{F}_1^t$, with potential function $p_{\mu^t}$. For $x\in\mathfrak{b}_{1,1}^t$, condition (\ref{eqn:meshcondition}) with $m_0=1$ together with (\ref{eqn:prepoteq}) implies \[-p_{\mu^t}(x) = \int_{\mathcal{F}_1^t} \log \delta_v(x,w) \, d\mu_t(w)
	= \frac{1}{d} \sum_{l=1}^d \log \delta^t_{1l}
	= \frac{1}{d} \big( r_1^t + \log \delta_{12}^t + \cdots + \log \delta_{1d}^t \big),\] and analogously for $x\in\mathfrak{b}_{1,i}^t$ for $2\le i\le q$. (Note that this does \emph{not} assume that all $\mathfrak{b}_{1,i}^t$ have multiplicity $1$.) On the other hand, Frostman's Theorem (Theorem \ref{thm:Frostman}) says that $p_{\mu^t}(x)=p_{\mu^t}(y)$ for all $x,y\in\mathcal{F}_1^t$. We thus obtain the system of equations \begin{equation}\label{eqn:r1lim}\lim_{t\to\infty}\left(r_1^t+\log\delta_{12}^t+\dots+\log\delta_{1d}^t\right)\end{equation}\[\verteq\] \begin{equation}\label{eqn:r2lim}\lim_{t\to\infty}\left(\log\delta_{21}^t+r_2^t+\log\delta_{23}^t+\dots+\log\delta_{2d}^t\right)\end{equation}\[\verteq\]\[\vdots\]\[\verteq\] \begin{equation}\label{eqn:rqlim}\lim_{t\to\infty}\left(\log\delta_{q1}^t+\dots+\log\delta_{q(q-1)}^t+r_q^t+\log\delta_{q(q+1)}^t+\dots+\log\delta_{qd}^t\right).\end{equation}Applying (\ref{eqn:diskinside}) to the first pair of limits in our system, namely (\ref{eqn:r1lim}) and (\ref{eqn:r2lim}), we see that forcibly, \begin{equation*}\label{eqn:r1r2}\lim_{t\to\infty}r_2^t=\lim_{t\to\infty}\log\delta_{12}^t,\end{equation*} i.e., \[\lim_{t\to\infty}r_2^t=\lim_{t\to\infty}r_1^t.\] From this we obtain $\mathcal{X}_1=\mathcal{X}_2$ and the index $i=2$ analogue of (\ref{eqn:diskinside}), namely \begin{equation}\label{eqn:diskinside2}\lim_{t\to\infty}\delta_{2l}^t\ge\lim_{t\to\infty}\delta_{jl}^t\textup{ for all }j\in\mathcal{X}_2\textup{ and all }1\le l\le d.\end{equation} Working in a similar manner with (\ref{eqn:r2lim}) and the analogous expression for $r_3^t$ (if applicable) and applying (\ref{eqn:diskinside2}), one gets \[\lim_{t\to\infty}r_3^t=\lim_{t\to\infty}r_2^t.\] Continuing this way for all equations in (\ref{eqn:r1lim})--(\ref{eqn:rqlim}), we find that \begin{equation*}\label{eqn:limconclusion}\lim_{t\to\infty}r_i^t=\lim_{t\to\infty}r_j^t\end{equation*} for all $i,j\in\mathcal{X}_1$. This completes the proof of the subclaim.\qed

	\begin{prop}\label{prop:wtdcap} Let $\epsilon>0$ and let $d\ge 2$. There are positive constants $N=N(d,\epsilon)$, $\tau=\tau(d,\epsilon)$, and $\epsilon'=\epsilon'(d,\epsilon)$ with the following property. Let $v\in M_K^0\setminus\mathscr{S}_d$, let $f(z)\in\CC_v[z]$ be a monic polynomial of degree $d$ having bad reduction at $v$, and let $g_v>0$ be the splitting radius of $f$. Let $T\subseteq\CC_v$ be a finite set. If $n=|T|\ge N$ and \[\fr{1}{n}\sum_{P_i\in T}\hat{\lambda}_v(P_i)\le\tau g_v,\] then \begin{equation}\label{eqn:equid}\log d_v(T)\le\epsilon g_v;\end{equation} if moreover $T$ fails to be $\epsilon$-equidistributed, then \begin{equation}\label{eqn:nonequid}\log d_v(T)\le-\epsilon'g_v.\end{equation}\end{prop}
	
	\begin{proof} Later in the proof, we will fix $0<\zeta<1$ sufficiently small, and then
		an integer $m\geq 1$ sufficiently large to guarantee (among other things) that
		$d^{-\lfloor m/2 \rfloor} < \zeta$. Then we will set $\tau=d^{-m}$. With that comment in place, we first prove (\ref{eqn:nonequid}). Let $v$, $f$, and $g_v$ be as in the
		statement of the proposition, and let $T\subseteq\CC_v$ be a nonempty finite set, and suppose $m\ge1$ is such that \begin{equation}\label{eqn:avg}\fr{1}{|T|}\sum_{P_i\in T}\hat{\lambda}_v(P_i)\le\fr{1}{d^m}g_v.\end{equation} We partition $T$ into three subsets based on local canonical height: $T_1$ will correspond to `small height points,' and $T_2$ and $T_3$ will correspond to `large height points.' Most points will lie in $T_1$. The large points are subdivided into $T_2$ and $T_3$ because the relationship between the local canonical height of a point and its distance to the filled Julia set of $f$ differs based on whether a point lies in $\mathcal{E}$. Suppose $\zeta$ is such that \begin{equation}\label{eqn:mzeta} d^{-\lfloor m/2\rfloor}<\zeta<1.\end{equation} By (\ref{eqn:avg}), there is a subset $T_1\subseteq T$ of size at least $(1-\zeta)|T|$ such that for each $P_i\in T_1$, \[\hat{\lambda}_v(P_i)\le\fr{1}{\zeta d^m}g_v.\] Assume $T_1$ contains all such points. Our condition on $\zeta$ and $m$ tells us that \begin{equation}\label{eqn:T1cond} T_1\subseteq\mathcal{E}_{\lceil m/2\rceil}.\end{equation} Let $T'=T\setminus T_1$, let $T_2$ be the set of elements $P_i\in T'$ such that \[\hat{\lambda}_v(P_i)\le g_v\] (equivalently, $T_2=\mathcal{E}\cap T'$),
		and let $T_3=T'\setminus\mathcal{E}$ be the set of elements $P_i\in T'$ such that \[\hat{\lambda}_v(P_i)>g_v.\] For disjoint nonempty finite sets $A,B\subseteq\mathbf{A}_v^1$, write \[(A,B)=\sum_{x\in A}\sum_{y\in B}\log\delta_v(x,y),\] and let \[(A,A)=\frac{1}{2}\sum_{x\ne y\in A}\log\delta_v(x,y).\] Let $|T|=n$. Then since $T_1,T_2,T_3$ are pairwise disjoint, \begin{equation}\label{eqn:dvT}\log d_v(T)=\fr{1}{n(n-1)}\sum_{x\ne y\in T}\log\delta_v(x,y)=\fr{2}{n(n-1)}\sum_{1\le i\le j\le 3}(T_i,T_j).\end{equation} As $v\notin\mathscr{S}_d$ and hence $\gamma(\mathcal{E})=\textup{exp}(g_v)$, it follows from the fact that $T_1\cup T_2\subseteq\mathcal{E}$ that \[\fr{1}{(\max\{1,|T_2|\})(|T_1|+\frac{1}{2}(|T_2|-1))}\sum_{i=1}^2(T_2,T_i)\le g_v,\] and thus, since $|T_2|/n\le\zeta$, \begin{equation}\begin{split}\label{eqn:T2}\fr{1}{n(n-1)}\sum_{i=1}^2(T_2,T_i)&\le \zeta\cdot\left(\frac{1}{\max\{1,|T_2|\}}\right)\cdot\left(\frac{1}{n-1}\right)\sum_{i=1}^2(T_2,T_i)\\&\le\zeta\cdot\fr{1}{(\max\{1,|T_2|\})(|T_1|+\frac{1}{2}(|T_2|-1))}\sum_{i=1}^2(T_2,T_i)\\&\le\zeta g_v.\end{split}\end{equation} Moreover, by (\ref{eqn:avg}), \begin{equation}\label{eqn:T3avg}\begin{split}\frac{1}{n}\sum_{P_i\in T_3}\hat{\lambda}_v(P_i)\le \frac{1}{n}\sum_{P_i\in T}\hat{\lambda}_v(P_i)\le\frac{g_v}{d^m}.\end{split}\end{equation} By the monicity hypothesis on $f$ and the fact that $v\in M_K^0\setminus\mathscr{S}_d$, we have that each point $P$ in the set \[W=\{Q\in\CC_v\mid \hat{\lambda}_v(Q)>g_v\}\] is at distance exactly $\text{exp}(\hat{\lambda}_v(P))$ from every point in $\mathbb{C}_v\setminus W$ \cite[Lemmas 2.1 and 2.2]{Ingram:PCF}. This implies that for any $x\in T_3$ and any $y\in T$, \begin{equation}\label{eqn:deltaht} \log\delta_v(x,y)\le \max\{\hat{\lambda}_v(x),\hat{\lambda}_v(y)\}.\end{equation} Therefore, it follows from (\ref{eqn:avg}), (\ref{eqn:T3avg}) and (\ref{eqn:deltaht}) (and the fact that $|T_3|/n\le\zeta$) that \begin{equation}\label{eqn:T3}\begin{split}\frac{1}{n(n-1)}\sum_{i=1}^3(T_3,T_i)&\le \frac{1}{n(n-1)}\left[\left(\sum_{\substack{x\in T_3\\y\in T_1\cup T_2}}\hat{\lambda}_v(x)+\hat{\lambda}_v(y)\right)+\left(\sum_{\substack{x\in T_3\\y\in T_3}}\hat{\lambda}_v(x)+\hat{\lambda}_v(y)\right)\right]\\&=\frac{1}{n(n-1)}\left[|T_1\cup T_2|\sum_{x\in T_3}\hat{\lambda}_v(x)+|T_3|\sum_{y\in T_1\cup T_2}\hat{\lambda}_v(y)+2|T_3|\sum_{z\in T_3}\hat{\lambda}_v(z)\right]\\&\le\frac{1}{n-1}\left(\sum_{x\in T_3}\hat{\lambda}_v(x)+2|T_3|\left(\frac{1}{n}\sum_{y\in T}\hat{\lambda}_v(y)\right)\right)	\\&\le\frac{1}{n-1} \bigg(n\cdot\frac{g_v}{d^m} + 2|T_3| \frac{g_v}{d^m}\bigg)\\&\le\bigg(\frac{n(1+2\zeta) }{n-1}\bigg) \frac{g_v}{d^m}.\end{split}\end{equation} If $\zeta$ is sufficiently small (which can be accomplished without violating (\ref{eqn:mzeta}) by assuming that $m$ is sufficiently large), and if $T$ is not $\epsilon$-equidistributed, then $T_1$ fails to be $\epsilon/2$-equidistributed. (One readily works out that $\zeta=\epsilon/(2d)$ makes this implication hold.) By (\ref{eqn:T1cond}), it thus follows immediately from Proposition \ref{prop:uniformeq} that if $\zeta\ll_{d,\epsilon}1$ and $m\gg_{d,\epsilon,\zeta}1$, then there is an $\epsilon_2=\epsilon_2(d,\epsilon)>0$ such that if $n\gg_{m,d,\epsilon}1$, then \[\fr{1}{|T_1|(|T_1|-1)}(T_1,T_1)\le-\epsilon_2g_v,\] and so \begin{equation}\label{eqn:T1}\fr{1}{n(n-1)}(T_1,T_1)\le-\epsilon_2(1-2\zeta)^2g_v\end{equation} for $n\gg_{m,d,\epsilon,\zeta}1$. Combining (\ref{eqn:T2}), (\ref{eqn:T3}), and (\ref{eqn:T1}), we obtain \begin{equation*}\begin{split}\log d_v(T)&=\fr{1}{n(n-1)}\sum_{x\ne y\in T}\log\delta_v(x,y)\\&=\frac{2}{n(n-1)}\sum_{1\le i\le j\le 3}(T_i,T_j)\\&\le -\epsilon'g_v\end{split}\end{equation*} where \[ \epsilon' = \epsilon'(d,\epsilon) = 
		2\left(\epsilon_2(1-2\zeta)^2 - \zeta - \frac{(1+2\zeta) n}{(n-1)d^m}\right).\] Letting $\zeta\ll_{d,\epsilon}1$ and $m\gg_{d,\epsilon,\zeta}1$ (which, again, we can do without violating (\ref{eqn:mzeta})), as well as $n\gg_{d,\epsilon,m,\zeta}1$ completes the proof of (\ref{eqn:nonequid}), with $\tau=\frac{1}{d^m}$.  We emphasize that $\epsilon'$ may be assumed to be positive under these conditions.
		
		Turning to prove (\ref{eqn:equid}), we first note that by (\ref{eqn:capEm}), we have $-I(\mu_{m,T})<\epsilon g_v/4$ for $m\gg_{d,\epsilon}1$. Moreover, Claim (\ref{eqn:weld2}) and our condition (\ref{eqn:T1cond}) immediately imply that if $n\gg_{d,\epsilon,m}1$, then \[ \frac{2}{|T_1| ( |T_1| -1)} (T_1, T_1)
		= \log d_v(T_1) \leq \frac{\epsilon}{4} g_v - I(\mu_{m,T}).\] Therefore, for $m\gg_{d,\epsilon}1$ and $n\gg_{d,\epsilon,m}1$, \[\frac{2}{|T_1|(|T_1|-1)}(T_1,T_1)\le\frac{\epsilon}{2} g_v,\] and hence \[\frac{2}{n(n-1)}(T_1,T_1)\le\frac{\epsilon}{2} g_v\] for $m\gg_{d,\epsilon}1$ and $n\gg_{d,\epsilon,m}1$. Combining this with (\ref{eqn:T2}) and (\ref{eqn:T3}) and applying (\ref{eqn:dvT}) yields \begin{equation*}\begin{split} \log d_v(T)&=\fr{1}{n(n-1)}\sum_{x\ne y\in T}\log\delta_v(x,y)\\&=\frac{2}{n(n-1)}\sum_{1\le i,j\le3}(T_i,T_j)\\&\le2\bigg(\bigg(\frac{(1+2\zeta) n}{n-1}\bigg)\frac{1}{d^m}+\zeta+\frac{\epsilon}{4}\bigg)g_v,\end{split}\end{equation*} and so the proof of (\ref{eqn:equid}) is concluded by taking $\zeta\ll_{d,\epsilon}1$, $m\gg_{d,\epsilon,\zeta}1$, and $n\gg_{d,\epsilon,m}1$. Altogether, we thus see that (\ref{eqn:equid}) and (\ref{eqn:nonequid}) hold simultaneously for $\zeta\ll_{d,\epsilon}1$, $m\gg_{d,\epsilon,\zeta}1$, and $n\gg_{d,\epsilon,m,\zeta}1$.\end{proof}
	\begin{proof}[Proof of Theorem \ref{thm:globaleq2}] Let $\kappa>0$, let $\epsilon>0$, let $\xi\in (0,1)$, and without loss of generality (as a field of definition for the appropriate conjugation will have degree at most $d-1$ over $K$), suppose that $f$ is monic. Suppose $T\subseteq K$ is a nonempty finite set such that \begin{equation}\label{eqn:kappabound1}\frac{1}{|T|}\sum_{P_i\in T}\hat{h}_f(P_i)\le\kappa h_{\textup{crit}}(f)\end{equation} and that $f$ satisfies \begin{equation}\label{eqn:xibound}\sum_{v\in M_K^0\setminus\mathscr{S}_d}r_v\lambda_{\textup{crit},v}(f)\ge(1-\xi)h_{\textup{crit}}(f).\end{equation} We will show that if $\kappa\ll_{d,\epsilon}1, \xi\ll_{d,\epsilon}1$, and $\kappa/\xi\ll_{d,\epsilon}1$, then there exist $N$ and $M$ (independent of $f$ and $T$) as in the statement of the theorem. Let $S_1$ be the set of places $v\in M_K^0\setminus\mathscr{S}_d$ of bad reduction for $f$ such that \begin{equation}\label{eqn:kappabound2}\frac{1}{|T|}\sum_{P_i\in T}\hat{\lambda}_v(P_i)\le\frac{\kappa}{\xi}\lambda_{\textup{crit},v}(f),\end{equation} and let $S_2=M_K\setminus S_1$. Let $S_{1,1}$ be the set of places in $S_1$ where $T$ is $\epsilon$-equidistributed, and let $S_{1,2}=S_1\setminus S_{1,1}$. Let $S$ be the set of places $v\in M_K^0\setminus\mathscr{S}_d$ such that $T$ is not $\epsilon$-equidistributed at $v$ (so that $S_{1,2}=S_1\cap S$), and let $0\le\zeta\le1$ be minimal subject to the condition that \begin{equation}\label{eqn:deltabound}\sum_{v\in S}r_v\lambda_{\textup{crit},v}(f)\ge (1-\zeta)h_{\textup{crit}}(f).\end{equation} Our goal is to prove that given any $0<\delta<1$, if $\kappa,\xi,\kappa/\xi\ll_{d,\epsilon,\delta}1$, then $\zeta\ge\delta$; in other words, $M_K^0\setminus(\mathscr{S}_d\cup S)$ forms a $\delta$-slice $\Sigma$ of places of $M_K^0\setminus\mathscr{S}_d$ such that $T$ is $\epsilon$-equidistributed for all $v\in\Sigma$. From (\ref{eqn:kappabound2}), it follows that \begin{equation}\label{eqn:kappabound3} \sum_{v\in M_K^0\setminus(\mathscr{S}_d\cup S_1)} r_v\lambda_{\textup{crit},v}(f)\le\xi h_{\textup{crit}}(f);\end{equation} otherwise, by the definition of $S_1$ given in (\ref{eqn:kappabound2}), we obtain \begin{equation*}\begin{split} \sum_{v\in M_K^0\setminus(\mathscr{S}_d\cup S_1)} \frac{1}{|T|}\sum_{P_i\in T}r_v\hat{\lambda}_v(P_i) &> \frac{\kappa}{\xi}\sum_{v\in M_K^0\setminus(\mathscr{S}_d\cup S_1)}r_v\lambda_{\textup{crit},v}(f)\\&>\kappa h_{\textup{crit}}(f),\end{split}\end{equation*} contradicting (\ref{eqn:kappabound1}). Similarly, \begin{equation*} \sum_{v\in M_K^\infty\cup\mathscr{S}_d}r_v\lambda_{\textup{crit},v}(f)>\xi h_{\textup{crit}}(f)\end{equation*} contradicts (\ref{eqn:xibound}), so in fact \begin{equation}\label{eqn:xibound2}\sum_{v\in M_K^\infty\cup\mathscr{S}_d}r_v\lambda_{\textup{crit},v}(f)\le\xi h_{\textup{crit}}(f).\end{equation} Inequalities (\ref{eqn:kappabound3}) and (\ref{eqn:xibound2}) yield \begin{equation}\label{eqn:xiboundS2}\begin{split}\sum_{v\in S_2}r_v\lambda_{\textup{crit},v}(f)&=\sum_{v\in M_K^0\setminus(\mathscr{S}_d\cup S_1)}r_v\lambda_{\textup{crit},v}(f)+\sum_{v\in M_K^\infty\cup\mathscr{S}_d}r_v\lambda_{\textup{crit},v}(f)\\&\le2\xi h_{\textup{crit}}(f).\end{split}\end{equation} Combining this with (\ref{eqn:deltabound}) and partitioning $S$ as $(S_1\cap S)\cup(S_2\cap S)=S_{1,2}\cup(S_2\cap S)$ gives \begin{equation*}\begin{split}\sum_{v\in S_{1,2}}r_v\lambda_{\textup{crit},v}(f)&= \sum_{v\in S}r_v\lambda_{\textup{crit},v}(f)-\sum_{S_2\cap S}r_v\lambda_{\textup{crit},v}(f)\\&\ge(1-\zeta-2\xi)h_{\textup{crit}}(f).\end{split}\end{equation*} Let $\epsilon_1>0$. By Proposition \ref{prop:wtdcap}, if $\kappa/\xi\ll_{d,\epsilon_1}1$, $|T|\gg_{d,\epsilon_1}1$, then for all $v\in S_{1,1}$, \[\log d_v(T)\le\epsilon_1g_v\] and there is an $\epsilon_2=\epsilon_2(d,\epsilon)>0$ such that for all $v\in S_{1,2}$, \[\log d_v(T)\le-\epsilon_2g_v.\] From this we obtain (applying Equation (\ref{eqn:splittingradiusmonic}) to deduce the second equality) \begin{equation}\begin{split}\label{eqn:S1}0=\sum_{v\in M_K}r_v\log d_v(T)&\le\sum_{v\in S_{1,1}}\epsilon_1r_vg_v+\sum_{v\in S_{1,2}}-\epsilon_2r_vg_v+\sum_{v\in S_2}r_v\log d_v(T)\\&=\sum_{v\in S_{1,1}}\epsilon_1r_v\lambda_{\textup{crit},v}(f)+\sum_{v\in S_{1,2}}-\epsilon_2r_v\lambda_{\textup{crit},v}(f)+\sum_{v\in S_2}r_v\log d_v(T)\\&\le \epsilon_1h_{\textup{crit}}(f)-\epsilon_2(1-\zeta-2\xi)h_{\textup{crit}}(f)+\sum_{v\in S_2}r_v\log d_v(T).\end{split}\end{equation} On the other hand, since $f$ is monic, we have that for all $v\in M_K$, \begin{equation}\label{eqn:Ingram} \log\max\{1,|P_i-P_j|_v\}\le\hat{\lambda}_v(P_i)+\hat{\lambda}_v(P_j)+\lambda_{\textup{crit},v}(f)+C_v,\end{equation} where $C_v$ is an $M_K$-constant supported only on $v\in M_K^\infty\cup\mathscr{S}_d$ that depends only on $d$ \cite[Lemma 2.1]{Ingram:PCF}. Applying (\ref{eqn:Ingram}), (\ref{eqn:kappabound1}), and (\ref{eqn:xiboundS2}) respectively for each of the following three inequalities, we have  \begin{equation*}\label{eqn:S2}\begin{split}\sum_{v\in S_2}r_v\log d_v(T)&\le\sum_{v\in S_2}r_v\cdot\frac{1}{|T|}\sum_{P_i\in T}\left(2\hat{\lambda}_v(P_i)+\lambda_{\textup{crit},v}(f)\right)+\sum_{v\in S_2}r_vC_v\\&\le2\kappa h_{\textup{crit}}(f)+\sum_{v\in S_2}r_v\lambda_{\textup{crit},v}(f)+\sum_{v\in S_2}r_vC_v\\&\le2\kappa h_{\textup{crit}}(f)+2\xi h_{\textup{crit}}(f)+\eta\end{split}\end{equation*} for $\eta=\eta(d)$. Combining this with (\ref{eqn:S1}) gives \begin{equation}\label{eqn:S1conclusion} 0 \leq (\epsilon_1 - \epsilon_2(1-\zeta-2\xi) + 2\kappa + 2\xi) h_{\textup{crit}}(f) + \eta .\end{equation} To conclude the proof, fix $0<\delta<1$. By (\ref{eqn:S1conclusion}), we see that if $h_{\textup{crit}}(f)\gg_\eta1$ and $\epsilon_1,\kappa,\xi\ll_{\epsilon_2,\delta}1$, then $\zeta\ge\delta$, i.e., $M_K^0\setminus(S\cup\mathscr{S}_d)$ forms a $\delta$-slice $\Sigma$ of places in $M_K^0\setminus\mathscr{S}_d$ such that $T$ is $\epsilon$-equidistributed at all $v\in\Sigma$. As $\epsilon_2$ depends only on $d$ and $\epsilon$, we may indeed assume that $\epsilon_1\ll_{\epsilon_2,\delta}1$ while also assuming $\kappa/\xi\ll_{d,\epsilon}1$ and $|T|\gg_{d,\epsilon}1$. Moreover, noting that $f$ satisfies (\ref{eqn:xibound}) with any $\xi'\ge\xi$ replacing $\xi$, we also see that we may assume that $\kappa/\xi\to 0$ as $\kappa,\xi\to0$. This completes the proof.\end{proof}

	\section{An easy case}
	
	In this section we show that when one restricts consideration to maps $f\in K[z]$ satisfying certain local hypotheses, proving Theorems \ref{thm:UBCpolys} and \ref{thm:dynLangpolys} is relatively straightforward. In particular, given $s\in\ZZ_{>0}$ and $0<\xi<1$, if a map $f\in K[z]$ of degree $d\ge2$ has a set of at most $s$ places that are either bad or archimedean, and contribute a total of at least $\xi h_{\textup{crit}}(f)$ to $h_{\textup{crit}}(f)$, then $f$ has at most $B=B(d,s,\xi)$ preperiodic points contained in $K$. We apply this idea to the special case where the set of places in question is $M_K^\infty\cup\mathscr{S}_d$. Combining this proposition with Theorem \ref{thm:globaleq2} leads to a much more natural form of Theorem \ref{thm:globaleq2}, where the conditions involving $\xi$ and $M$ are dropped. This corresponds to Theorem \ref{thm:globaleq}.
	
	\begin{prop}\label{prop:Mxi} Let $d\ge2$, let $0<\xi<1$, let $s$ be a positive integer, and let $M\in\mathbb{R}_{>0}$. There are constants $N=N(d,[K:F],\xi,s,M)$ and $\kappa=\kappa(d,[K:F],\xi,s,M)>0$ such that if $f(z)\in K[z]$ is a polynomial of degree $d\ge2$ (assumed to be non-isotrivial if $K$ is a function field), $S$ is a finite set of places of $M_K$ with $|S|=s$, and $T\subseteq K$ is a finite set such that $|T|\ge N$ and \[\frac{1}{|T|}\sum_{P_i\in T}\hat{h}_f(P_i)\le\kappa\max\{1,h_{\textup{crit}}(f)\},\] then $h_{\textup{crit}}(f)\ge M$ and \[\sum_{v\in M_K\setminus S}r_v\lambda_{\textup{crit},v}(f)\ge(1-\xi)h_{\textup{crit}}(f).\]\end{prop} To prove Proposition \ref{prop:Mxi}, we require a result governing the minimal distance of a point of small local canonical height to a point of $f^{-4}(0)$, assuming $f$ takes an appropriate form.
	
	\begin{prop}\cite[Corollary 3.4, Proposition 4.3]{Looper:mincanht}\label{prop:squish} Let $d\ge2$, and let $\epsilon\ge0$. Suppose that $K$ is either a number field or a function field $K/k(t)$ where $\textup{char}(K)=0$ or $\textup{char}(K)>d$. There is a constant $\eta=\eta(d,\epsilon)$ with the following property. Let $f(z)\in K[z]$ be a monic polynomial of degree $d$, and let $v\in M_K$. Suppose $\beta\in\mathcal{K}_v$ and $\alpha\in\CC_v$. Then $\alpha\in {D(\beta,(1+\epsilon)e^{\lambda_{\textup{crit},v}(f)})}$ implies that every $y\in f^{-4}(\alpha)$ satisfies \[\min_{\gamma\in f^{-4}(\beta)}\log|y-\gamma|_v\le-\frac{1}{d-1}\lambda_{\textup{crit},v}(f)+\eta.\]\end{prop}
	
	\begin{rmk}The normal form for $f$ used in \cite{Looper:mincanht} has $\beta=0$ as a fixed point of $f$, and lead coefficient $1/d$. However, the proof equally applies to the more general case of $\beta\in\mathcal{K}_v$, and the change in the lead coefficient can be absorbed into $\eta$. We also note that in \cite{Looper:mincanht}, the radius of the disk containing $\alpha$ is $e^{\lambda_{\textup{crit},v}(f)}$ rather than $(1+\epsilon)e^{\lambda_{\textup{crit},v}(f)}$, and one considers $f^{-3}(\beta)$ rather than $f^{-4}(\beta)$. The change from $f^{-3}(\beta)$ to $f^{-4}(\beta)$ accommodates the change from radius $e^{\lambda_{\textup{crit},v}(f)}$ to radius $(1+\epsilon)e^{\lambda_{\textup{crit},v}(f)}$, as $(1+\epsilon)e^{\lambda_{\textup{crit},v}(f)}<e^{d\lambda_{\textup{crit},v}(f)}$ when $\lambda_{\textup{crit},v}(f)$ is sufficiently large. \end{rmk}
	
	\begin{proof}[Proof of Proposition \ref{prop:Mxi}] Let $\epsilon\ge 0$, let $d\ge2$, let $\kappa>0$, let $0<\xi<1$, let $s\ge 1$ be an integer, let $M\in\mathbb{R}_{>0}$, and without loss of generality, assume $f$ is monic. Suppose $T\subseteq K$ is a finite set such that \begin{equation}\label{eqn:kappahyp}\frac{1}{|T|}\sum_{P_i\in T}\hat{h}_f(P_i)\le\kappa\max\{1, h_{\textup{crit}}(f)\}.\end{equation} We first aim to show that there is an $N_1=N_1(\xi,s,d,[K:F])$ such that if $\kappa\ll_{\xi,s,d,[K:F]}1$ and \begin{equation}\label{eqn:xihyp}\sum_{v\in S}r_v\lambda_{\textup{crit},v}(f)\ge\xi h_{\textup{crit}}(f)\end{equation} for some nonempty set $S$ of places of $K$ with $|S|=:s$, then $|T|\le N_1$. We will then show that there is an $N_2=N_2(M,[K:F],d)$ such that if $\kappa\ll_{M,[K:F],d}1$ and $h_{\textup{crit}}(f)\le M$, then $|T|\le N_2$. Taking $N=\max\{N_1,N_2\}$ for the given $s,\xi,M,d,[K:F]$ then completes the proof of Proposition \ref{prop:Mxi}.
		
		By (\ref{eqn:xihyp}), there is a $v_0\in S$ such that \[r_{v_0}\lambda_{\textup{crit},v_0}(f)\ge\frac{\xi}{s}h_{\textup{crit}}(f).\] Let $S_0=M_K^\infty\cup\mathscr{S}_d\cup\{v_0\}$, with $|S_0|=s_0$. A fortiori, \begin{equation}\label{eqn:s0}\sum_{v\in S_0}r_v\lambda_{\textup{crit},v}(f)\ge\frac{\xi}{s}h_{\textup{crit}}(f).\end{equation}

		\textbf{Claim}: Let $0<\epsilon'<\xi/(ss_0)$. There is a constant $\kappa'=\kappa'(\epsilon',[K:F])>0$ and a partition $S_0=V_0\sqcup W_0$ of $S_0$ such that: \begin{enumerate}[label=(\roman*), ref=(\roman*)] \item\label{eqn:item1} $V_0\ne\emptyset$ \item\label{eqn:item2} For all $v\in V_0$ and all $\alpha\in K$ with $\hat{h}_f(\alpha)\le\kappa' h_{\textup{crit}}(f)$, \begin{equation}\label{eqn:belongtodisk}\hat{\lambda}_v(\alpha)\le \lambda_{\textup{crit},v}(f).\end{equation} \item\label{eqn:item3} For all $v\in W_0$, \begin{equation}\label{eqn:epsilon'} \lambda_{\textup{crit},v}(f)\le\epsilon' h_{\textup{crit}}(f).\end{equation}\end{enumerate}

		Proof of Claim: We claim that we may choose any $0<\kappa'\le\epsilon'/[K:F]$ in order to produce such a $V_0$ and $W_0$. Fix $0<\epsilon'<\xi/(ss_0)$, and let \[W_0=\{v\in S_0:\lambda_{\textup{crit},v}(f)\le\epsilon'h_{\textup{crit}}(f)\}.\] Let $\alpha\in K$. Suppose that $v\in V_0:=S_0\setminus W_0$. If $0<\kappa'\le\epsilon'/[K:F]$, then \[\hat{\lambda}_v(\alpha) \leq [K:F] r_v \hat{\lambda}_v(\alpha) \leq [K:F] \hat{h}_f(\alpha)
		\leq [K:F] \kappa' h_{\mathrm{crit}}(f) \leq \epsilon' h_{\mathrm{crit}}(f) < \lambda_{\mathrm{crit},v}(f) \]
		for all $\alpha\in K$ such that $\hat{h}_f(\alpha) \leq \kappa' h_{\mathrm{crit}}(f)$. Note that the final inequality holds because $v\not\in W_0$. This proves (\ref{eqn:belongtodisk}).

		Finally, the fact that $V_0=S_0\setminus W_0\ne\emptyset$ is immediate from (\ref{eqn:s0}) and the inequality $\epsilon'<\frac{\xi}{ss_0}$, since $W_0=S_0$ forces \begin{equation*}\frac{\xi}{s}h_{\textup{crit}}(f)=\sum_{v\in S_0}\frac{\xi}{ss_0}h_{\textup{crit}}(f)>\sum_{v\in S_0}\epsilon'h_{\textup{crit}}(f)\ge\sum_{v\in S_0}r_v\epsilon'h_{\textup{crit}}(f)\ge\sum_{v\in S_0}r_v\lambda_{\textup{crit},v}(f)\ge\frac{\xi}{s}h_{\textup{crit}}(f),\end{equation*} a contradiction. This proves the desired claim.\qed

		Returning to the proof of Proposition \ref{prop:Mxi}, suppose that $h_{\textup{crit}}(f)\ge1$. (We will address that case where $h_{\textup{crit}}(f)$ is bounded, and hence the case where $h_{\textup{crit}}(f)<1$, later.) Let $0<\epsilon'<\xi/(ss_0)$, from which it follows both that $\epsilon'<\frac{\xi}{s}h_{\mathrm{crit}}(f)$ and that we may let $\kappa'=\kappa'(\epsilon',[K:F])>0$ and $V_0,W_0$ be as in the Claim. Assume without loss that $\kappa$ as in (\ref{eqn:kappahyp}) satisfies \begin{equation}\label{eqn:kappa'} 2\kappa\le\kappa'.\end{equation} Clearly, at least $|T|/2$ elements $\alpha\in T$ must satisfy \begin{equation}\label{eqn:htbdkapprime}\hat{h}_f(\alpha)\le2\kappa h_{\textup{crit}}(f)\le\kappa'h_{\textup{crit}}(f),\end{equation} for otherwise, as $h_{\textup{crit}}(f)\ge1$, (\ref{eqn:kappahyp}) is violated. Let \begin{equation}\label{eqn:T'} T'=\{\alpha\in T:\hat{h}_f(\alpha)\le2\kappa h_{\textup{crit}}(f).\}\end{equation} By the Claim, (\ref{eqn:belongtodisk}) holds for all $\alpha\in T'$ and all $v\in V_0$. We further note that by \cite[Lemma 2.1]{Ingram:PCF} and the fact that we have assumed without loss that $f$ is monic, there is an $M_K$-constant $\{C_v'\}_{v\in M_K}$ depending only on $d$ such that for any $v\in M_K$, $\hat{\lambda}_v(\alpha_1),\hat{\lambda}_v(\alpha_2)\le\lambda_{\textup{crit},v}(f)$ implies that \begin{equation}\label{eqn:distance}\log|\alpha_1-\alpha_2|_v\le\lambda_{\textup{crit},v}(f)+C_v'.\end{equation} Thus if $\epsilon\ge0$ is such that $\log(1+\epsilon)>\max_{v\in M_K}C_v'$, then for any $v\in M_K$ and any $\beta_v\in\mathcal{K}_v$, the closed disk ${D\left(\beta_v,(1+\epsilon)e^{\lambda_{\textup{crit},v}(f)}\right)}$ in $\CC_v$ contains the set \[\mathcal{E}=\{\alpha\in\CC_v:\hat{\lambda}_v(\alpha)\le \lambda_{\textup{crit},v}(f)\}.\] Fix such an $\epsilon=\epsilon(d)>0$ (taken to be positive here for easy reuse in (\ref{eqn:notS0})). Let $W=f^{-4}(T')$, and let $v\in V_0$, so that by (\ref{eqn:htbdkapprime})/(\ref{eqn:T'}) and the definition of $V_0$, $T'\subseteq\mathcal{E}$. Fix $\beta\in\mathcal{K}_v$. By Proposition \ref{prop:squish} applied to $\epsilon$, we have that for any $\alpha\in T'$ and any $y\in f^{-4}(\alpha)$, there is a $\gamma_y\in f^{-4}(\beta)$ such that \[\log|y-\gamma_y|_v\le-\frac{1}{d-1}\lambda_{\textup{crit},v}(f)+\eta,\] where $\eta=\eta(d,\epsilon)$. Assume that each $\gamma_y$ is chosen to have minimal distance from $y$ over the various $\gamma\in f^{-4}(\beta)$. For each $\gamma_i\in f^{-4}(\beta)$, write $\mathfrak{U}_i=\{y\in W:\gamma_y=\gamma_i\}$. As there are at most $d^4$ such sets $\mathfrak{U}_i$, whose union is $W$, the pigeonhole principle implies that there is some $i$ such that $|\mathfrak{U}_i|\ge|W|/d^4$. Since each $y\in\mathfrak{U}_i$ lies at distance at most \begin{equation*}\textup{exp}\left(\eta-\frac{1}{d-1}\lambda_{\textup{crit},v}(f)\right)\end{equation*} from $\gamma_i$, $\mathfrak{U}_i$ is contained in a disk of that radius in $\mathbb{C}_v$. Write $U_1=\mathfrak{U}_i$ for this $i$, so that we have \begin{equation*}\label{eqn:U1size}|U_1|\ge |W|/d^4.\end{equation*} Note that since $W=f^{-4}(T')$ and in general $\hat{h}_f(f(x))=d\hat{h}_f(x)$, inequalities (\ref{eqn:kappa'}) and (\ref{eqn:T'}) imply that for each $\alpha\in W$,  \begin{equation}\label{eqn:U1htbd}\hat{h}_f(\alpha)\le\frac{2\kappa}{d^4}h_{\textup{crit}}(f)<\kappa'h_{\mathrm{crit}}(f).\end{equation} In particular, (\ref{eqn:U1htbd}) holds for each $\alpha\in U_1$, from which it follows by the Claim that (\ref{eqn:belongtodisk}) holds at any $v'\in V_0$. We then apply Proposition \ref{prop:squish} again to $U_1$ (in lieu of $W$) and to any $v'\ne v\in V_0$ (supposing such a place exists) in order to obtain a subset $U_2\subseteq U_1$ with \begin{equation}\label{eqn:U2size}|U_2|\ge |U_1|/d^4\end{equation} such that $U_2$ is contained in a disk of radius at most \begin{equation}\label{eqn:U2}\textup{exp}\left(\eta-\frac{1}{d-1}\lambda_{\textup{crit},v'}(f)\right)\end{equation} in $\CC_{v'}$, and a fortiori $\hat{h}_f(\alpha)<2\kappa h_{\mathrm{crit}}(f)\le\kappa'h_{\textup{crit}}(f)$ for all $\alpha\in U_2$. Continuing in this manner for each $v\in V_0$, we obtain a set $U=U_{|V_0|}\subseteq W$ with \begin{equation}\label{eqn:Usize}|U|\ge |W|/d^{4|V_0|}\ge|T'|/d^{4|V_0|}\ge|T|/2(d^{4|V_0|})\end{equation} such that for all $v\in V_0$, $U$ is contained in a disk of radius at most \begin{equation}\label{eqn:Unrad}\textup{exp}\left(\eta-\frac{1}{d-1}\lambda_{\textup{crit},v}(f)\right)\end{equation} in $\CC_v$, and \begin{equation}\label{eqn:kappahyp2}\hat{h}_f(\alpha)<2\kappa h_{\textup{crit}}(f)\end{equation} for all $\alpha\in U$.

		On the other hand, we trivially have by (\ref{eqn:kappahyp2}) that for all $\alpha\in U$ and all $v\in W_0$, \[\hat{\lambda}_v(\alpha)<2\kappa h_{\textup{crit}}(f)/r_v,\] from which it follows by combining with (\ref{eqn:Ingram}) and (\ref{eqn:epsilon'}) that for all $\alpha_1\ne\alpha_2\in U$ and all $v\in W_0$, \begin{equation}\label{eqn:W0}\begin{split}\log\max\{1,|\alpha_1-\alpha_2|_v\}&\le\hat{\lambda}_v(\alpha_1)+\hat{\lambda}_v(\alpha_2)+\lambda_{\textup{crit},v}(f)+C_v\\&<4\kappa h_{\textup{crit}}(f)/r_v+\lambda_{\textup{crit},v}(f)+C_v\\&\le (4\kappa/r_v+\epsilon')h_{\textup{crit}}(f)+C_v.\end{split}\end{equation}

		Now let $v\in M_K\setminus S_0$ be a place of bad reduction for $f$. Since \[g_v=\lambda_{\textup{crit},v}(f)\le [K:F] h_{\textup{crit}}(f)\] (recall that we have assumed without loss that $f$ is monic), Proposition \ref{prop:wtdcap} says that there are constants $N_0=N_0(d,\epsilon)\ge2$ and $\tau=\tau(d,\epsilon)$ such that if $|U|\ge N_0$ and the $\kappa$ in (\ref{eqn:kappahyp}) satisfies $\kappa\le\tau$, then we have \begin{equation}\label{eqn:notS0}\log d_v(U)\le\epsilon\lambda_{\textup{crit},v}(f).\end{equation} Assume without loss that $\kappa\le\tau$ and $|U|\ge N_0$.
		
		Finally, if $v\in M_K\setminus S_0$ is a place of good reduction for $f$, then again by the monicity of $f$, there is a $\beta\in\CC_v$ such that $\log\max\{1,|x-\beta|_v\}=\hat{\lambda}_v(x)$ for all $x\in\CC_v$. This implies (as $|U|\ge2$) that \begin{align*}\sum_{\substack{v\in M_K\setminus S_0\\ v\textup{ good}}}r_v\log d_v(U)&=\sum_{\substack{v\in M_K\setminus S_0\\ v\textup{ good}}}r_v\log\prod_{P_i\ne P_j\in U}|P_i-P_j|_v^{\frac{1}{|U|(|U|-1)}}\stepcounter{equation}\tag{\theequation}\label{eqn:gooodcase}\\&=\sum_{\substack{v\in M_K\setminus S_0\\ v\textup{ good}}}r_v\log\prod_{P_i\ne P_j\in U}|(P_i-\beta)-(P_j-\beta)|_v^{\frac{1}{|U|(|U|-1)}}\stepcounter{equation}\tag{\theequation}\label{eqn:gooodcase}\\&\le\sum_{\substack{v\in M_K\setminus S_0\\ v\textup{ good}}}r_v\log\prod_{P_i\ne P_j\in U}\max\{|P_i-\beta|_v,|P_j-\beta|_v\}^{\frac{1}{|U|(|U|-1)}}\\&\le\sum_{\substack{v\in M_K\setminus S_0\\ v\textup{ good}}}r_v\log\prod_{P_i\ne P_j\in U}(\max\{1,|P_i-\beta|_v\}\max\{1,|P_j-\beta|_v\})^{\frac{1}{|U|(|U|-1)}}\\&=\sum_{\substack{v\in M_K\setminus S_0\\ v\textup{ good}}}r_v\log\left(\prod_{P_i\in U}(\max\{1,|P_i-\beta|_v\}^2)^{1/|U|}\right)\\&=\sum_{\substack{v\in M_K\setminus S_0\\ v\textup{ good}}}\frac{2}{|U|}\sum_{P_i\in U}r_v\log\max\{1,|P_i-\beta|_v\}\\&=\frac{2}{|U|}\sum_{\substack{v\in M_K\setminus S_0\\ v\textup{ good}}}\sum_{P_i\in U}r_v\hat{\lambda}_v(P_i)\\&<4\kappa h_{\textup{crit}}(f),\end{align*} where the final inequality follows from (\ref{eqn:kappahyp2}).
		
		Combining (\ref{eqn:Unrad}) and (\ref{eqn:W0})--(\ref{eqn:gooodcase}), we obtain \begin{equation}\label{eqn:upbd}\begin{split}\sum_{v\in M_K}r_v\log d_v(U)<&\sum_{v\in V_0}r_v\left(-\frac{1}{d-1}\lambda_{\textup{crit},v}(f)+\eta+\log2\right)\\&+\sum_{v\in W_0}r_v(4\kappa/r_v+\epsilon')h_{\textup{crit}}(f)\\&+\sum_{v\in W_0}r_vC_v+\sum_{\substack{v\in M_K\setminus S_0\\ v\textup{ bad}}}r_v\epsilon\lambda_{\textup{crit},v}(f)\\&+4\kappa h_{\textup{crit}}(f).\end{split}\end{equation} Setting $C=\max_{v\in M_K}C_v$, inequalities (\ref{eqn:s0}) and (\ref{eqn:upbd}) yield the inequality \begin{equation}\label{eqn:finalineq}\begin{split}\sum_{v\in M_K}r_v\log d_v(U)<&\left(\left(-\frac{1}{(d-1)}\right)\frac{\xi}{s}+|W_0|(4\kappa+\epsilon')+\epsilon\right)h_{\textup{crit}}(f)\\&+|V_0|(\eta+\log2)+|W_0|C+4\kappa h_{\textup{crit}}(f)\\\le&\left(\left(-\frac{1}{(d-1)}\right)\frac{\xi}{s}+s\left(8\kappa+\epsilon'\right)+\epsilon\right)h_{\textup{crit}}(f)\\&+s(\eta+\log2+C),\end{split}\end{equation} the right-hand side of which is negative if $\kappa,\epsilon,\epsilon'\ll_{\xi,s,d}1$ and $h_{\textup{crit}}(f)\gg_{s(\eta+\log2+C)}1$, contradicting the product formula. We highlight that $\epsilon$ and $\epsilon'$ may be assumed to be arbitrarily small from the outset, and then $\kappa$ is chosen subsequently to satisfy $\kappa\ll_{\epsilon,\epsilon',\xi,s,d}1$. Moreover, by (\ref{eqn:Usize}), $|T|$ may be assumed to be large enough (in terms of $d$ and $|V_0|$) to ensure $|U|\ge N_0\ge 2$, making all inequalities from (\ref{eqn:notS0}) onward hold. Therefore under these conditions, the contradiction that we have reached implies that there is an $N_{1,1}=N_{1,1}(d,|V_0|)=N_{1,1}(s,d,[K:F])$ such that $|T|\le N_{1,1}$.

		All that is left to address is the case where $h_{\textup{crit}}(f)$ is bounded, say by $M'$. When $K$ is a number field, Northcott's Theorem along with the fact that $h_{\textup{crit}}(f)\asymp h_{\mathcal{M}_d}(f)$ by \cite{Ingram:critmod} (as discussed in the lines following Theorem \ref{thm:dynLangpolys}) immediately implies that there is an $N_{1,2}=N_{1,2}(M',[K:F],d)$ such that if \[\mathscr{F}(K,M',d)=\{f\in K[z]:f\textup{ has degree } d\textup{ and } h_{\textup{crit}}(f)\le M'\}\] and \[\mathcal{A}(K,M',d)=\sup_{f\in\mathscr{F}(K,M',d)}|\{\alpha\in K:\hat{h}_f(\alpha)\le\max\{1,h_{\textup{crit}}(f)\}|,\] then $\mathcal{A}(K,M',d)\le N_{1,2}$. Since we may assume without loss that $\kappa\le1$, this concludes the proof of the existence of $N_1=\max\{N_{1,1},N_{1,2}\}$ as in the sentence containing (\ref{eqn:xihyp}) when $K$ is a number field. Now suppose $K/k(t)$ is a function field (of characteristic either $0$ or $>d$), and that $f\in K[z]$ of degree $d$ is non-isotrivial. In this case $S_0=\{v_0\}$ (cf.~the definitions right before the Claim), and hence $V_0=\{v_0\}$ and $W_0=\emptyset$. We may set $\epsilon=0$ in the discussion below (\ref{eqn:distance}); the associated constant from Proposition \ref{prop:squish} is then $\eta(d,0)$. We have $\eta(d,0)=0$ \cite[pp.~1067--1068]{Looper:mincanht}. Consequently, (\ref{eqn:upbd}) becomes \begin{equation}\label{eqn:fnfldbd}\sum_{v\in M_K}r_v\log d_v(U)< r_{v_0}\left(-\frac{1}{d-1}\lambda_{\textup{crit},v_0}(f)\right)+\sum_{\substack{v\in M_K\setminus S_0\\v\textup{ bad}}}r_v\epsilon\lambda_{\textup{crit},v}(f)+4\kappa h_{\textup{crit}}(f)\end{equation} and (\ref{eqn:finalineq}) reads as \begin{equation*}\sum_{v\in M_K}\log d_v(U)<\left(\left(-\frac{1}{d-1}\right)\frac{\xi}{s}+\epsilon+4\kappa \right)h_{\textup{crit}}(f).\end{equation*} This leads to a contradiction when $\kappa,\epsilon\ll_{\xi,s,d}1$, since $f$ being non-isotrivial forces $h_{\textup{crit}}(f)\ne 0$; indeed, by \cite[Proposition 6.1]{Benedetto1}, isotriviality is equivalent to potential good reduction at all $v\in M_K$, which by (\ref{eqn:splittingradiusmonic}) is in turn equivalent to the condition $h_{\textup{crit}}(f)=0$. We thus see that the proof in the function field case does not depend on whether $h_{\textup{crit}}(f)$ is bounded. It follows that we may in fact set the $N_1$ referred to in (\ref{eqn:xihyp}) to be $N_1=\max\{N_{1,1},N_{1,2}\}$ in both the number field and the function field case. The existence of $N_2=N_2(M,[K:F],d)$ such that $\kappa\le1$ and $h_{\textup{crit}}(f)\le M$ implies $|T|\le N_2$ has just been addressed via Northcott's Theorem in the number field case. When $K$ is a function field, $h_{\textup{crit}}(f)\le M$ implies that $f$ has at most $Md(d-1)[K:F]$ places of bad reduction (cf.~inequality (\ref{eqn:unif})) and hence that there is a $v_0$ among those places such that $r_{v_0}\lambda_{\textup{crit},v_0}(f)\ge\frac{1}{Md(d-1)[K:F]}h_{\textup{crit}}(f)$. Inequality (\ref{eqn:fnfldbd}) then yields a contradiction when $\kappa,\epsilon\ll_{M,d,[K:F]}1$. Taking $N=\max\{N_1,N_2\}$ completes the proof.\end{proof}
	
	Combining Proposition~\ref{prop:Mxi} (with $S=\mathscr{S}_d\cup M_K^{\infty}$) and Theorem~\ref{thm:globaleq2} immediately implies the main quantitative equidistribution result in this paper. 
	
	\begin{thm}[Global quantitative equidistribution]\label{thm:globaleq} Let $d\ge2$, let $\epsilon>0$, and let $0<\delta<1$. There are constants $N$ and $\kappa>0$, depending only on $[K:F]$, $d$, $\epsilon$, and $\delta$, with the following property. If $f(z)\in K[z]$ is a polynomial of degree $d$, and $T\subseteq K$ is a finite set such that $|T|\ge N$ and \[\frac{1}{|T|}\sum_{P_i\in T}\hat{h}_f(P_i)\le\kappa\max\{1,h_{\textup{crit}}(f)\},\] then $T$ is $\epsilon$-equidistributed for a $\delta$-slice of bad places $v\in M_K^0\setminus\mathscr{S}_d$.\end{thm}
	
	\section{Differences of small points}\label{section:differences}
	
	This section is devoted to showing that elements of the form $P_i-P_j$, where $P_i,P_j$ have small canonical height relative to $f$, tend to have their prime support mostly contained within the set of places of bad reduction for $f$. (Of course, one must take the lead coefficient of $f$ into account.) In the special case where $P_i,P_j$ are preperiodic, this is a dynamical analogue of a parallel phenomenon in the setting of groups: for instance, $n$-th roots of unity remain distinct modulo primes not dividing $n$, and $n$-torsion points on elliptic curves remain distinct modulo primes of good reduction not dividing $n$.

	\begin{definition*} Let $\epsilon>0$, and let $f(z)\in K[z]$ be of degree $d\ge2$ with lead coefficient $a_d$. Let $S$ be the union of the set of places of good reduction for $f$ and $\mathscr{S}_d\cup M_K^\infty$. We say $\alpha\in K^*$ is \emph{$\epsilon$-adelically good} if \[\sum_{v\in S}r_v\left|\log|a_d\alpha^{d-1}|_v\right|\le\epsilon h_{\textup{crit}}(f).\] \end{definition*}
	
	\begin{rmk} As seen in the following proposition, our choice of $\alpha$ will consist of differences of points having small canonical height. It is readily verified that the adjustment by the lead coefficient $a_d$ has the effect of making the notion of $\epsilon$-adelically good for such $\alpha$ be invariant under conjugation by scaling. In other words, if $\mu(z)=\beta z$ for $\beta \in K$, then $\alpha=P_i-P_j$ is $\epsilon$-adelically good for $f$ if and only if $\mu(P_i)-\mu(P_j)$ is $\epsilon$-adelically good for $g:=\mu f\mu^{-1}$. Similarly, such $\alpha$ will also be invariant under conjugation by translation, as differences of pairs of points are left invariant under this kind of change of variable. \end{rmk}
	
	\begin{prop}\label{prop:adgoodbars} Let $\epsilon>0$, and let $d\ge2$ be an integer. There are constants $N=N(d,\epsilon,[K:F])$ and $\kappa=\kappa(d,\epsilon,[K:F])>0$ with the following property: if $f(z)\in K[z]$ is a polynomial of degree $d$, and $T\subseteq K$ is a set of $n\ge N$ points satisfying \[\frac{1}{n}\sum_{P_i\in T}\hat{h}_f(P_i)\le\kappa \max\{1,h_{\textup{crit}}(f)\},\] then at least $(1-\epsilon)(n(n-1))$ elements $(P_i,P_j)\in T^2\setminus\Delta(T^2)$ have the property that $P_i-P_j$ is $\epsilon$-adelically good.\end{prop}
	
	\begin{proof}Let $\epsilon>0$, and let $1>\xi>0$. Suppose without loss that $f(z)\in K[z]$ of degree $d\ge2$ is monic. Proposition \ref{prop:Mxi} says that there is a constant $\kappa=\kappa(d,\xi,[K:F])$ such that if $T\subseteq K$ of order $n$ satisfies \begin{equation}\label{eqn:kappachoice}\frac{1}{n}\sum_{P_i\in T}\hat{h}_f(P_i)\le\kappa\max\{1, h_{\textup{crit}}(f)\}\end{equation} and $|T|\gg_{d,\xi,[K:F]}1$, then \begin{equation}\label{eqn:xichoice}\sum_{v\in M_K^\infty\cup\mathscr{S}_d}r_v\lambda_{\textup{crit},v}(f)\le\xi h_{\textup{crit}}(f).\end{equation} We therefore assume without loss that $\kappa$ is chosen to have this property, so that (\ref{eqn:xichoice}) holds for any sufficiently large $T\subseteq K$ satisfying (\ref{eqn:kappachoice}). Let $T$ be such a set. Also by Proposition \ref{prop:Mxi}, we may assume without loss that $h_{\textup{crit}}(f)\ge1$, so that \begin{equation}\label{eqn:kappabound}\frac{1}{n}\sum_{P_i\in T}\hat{h}_f(P_i)\le\kappa h_{\textup{crit}}(f).\end{equation} Let $S$ be the union of the set of places of good reduction for $f$ and $M_K^\infty\cup\mathscr{S}_d$. Let $\zeta>\max\left\{1,\frac{1}{\epsilon}\right\}$, and let $S'$ be the set of places of $M_K\setminus S$ such that \begin{equation*}\label{eqn:zetabound}\frac{1}{n}\sum_{P_i\in T}\hat{\lambda}_v(P_i)\le\zeta\kappa \lambda_{\mathrm{crit},v}(f).\end{equation*} By (\ref{eqn:kappabound}), we have: \begin{equation}\label{eqn:S'slice} \sum_{v\in S'}r_v\lambda_{\mathrm{crit},v}(f)\ge \left(1-\frac{1}{\zeta}\right)h_{\mathrm{crit}}(f).\end{equation} By Proposition \ref{prop:wtdcap}, there is an $N=N(d,\epsilon)$ such that if $n\ge N$ and $\kappa\ll_{d,\epsilon,\zeta}1$, then for any $v\in S'$, \[\frac{1}{n(n-1)}\sum_{P_i\ne P_j\in T}\log|P_i-P_j|_v\le\frac{\epsilon^2}{2}g_v=\frac{\epsilon^2}{2}\lambda_{\textup{crit},v}(f),\] and hence \begin{equation}\label{eqn:avg1}\frac{1}{n(n-1)}\sum_{P_i\ne P_j\in T}\sum_{v\in S'}r_v\log|P_i-P_j|_v\le\frac{\epsilon^2}{2}\sum_{v\in S'}r_v\lambda_{\textup{crit},v}(f).\end{equation} Assume in what follows that $n\ge N$ and that $\kappa$ is sufficiently small for this to hold. On the other hand, we claim that for any $\tau>0$, there is a constant $C$ depending only on $d$ such that for any subset $W\subseteq T^2\setminus\Delta(T^2)$ with \begin{equation}\label{eqn:tauhtbd}\frac{1}{|W|}\sum_{(P_i,P_j)\in W}\hat{h}_f(P_i)+\hat{h}_f(P_j)\le\tau h_{\mathrm{crit}}(f),\end{equation} one has \begin{equation}\label{eqn:avg2}\frac{1}{|W|}\sum_{(P_i,P_j)\in W}\sum_{v\in S'}r_v\log|P_i-P_j|_v\ge-C-(\xi+1/\zeta+\tau) h_{\textup{crit}}(f).\end{equation} Indeed, from (\ref{eqn:Ingram}) we have that for all $v\in M_K$, \begin{equation}\label{eqn:C'}\log|P_i-P_j|_v\le\hat{\lambda}_v(P_i)+\hat{\lambda}_v(P_j)+\lambda_{\textup{crit},v}(f)+C'\end{equation} for a constant $C'$ depending only on $d$. When $v\in M_K^0\setminus\mathscr{S}_d$, the stronger inequality \begin{equation}\label{eqn:stronger}\log|P_i-P_j|_v\le\hat{\lambda}_v(P_i)+\hat{\lambda}_v(P_j)+\lambda_{\mathrm{crit},v}(f)\end{equation} holds. Combining with (\ref{eqn:xichoice}), (\ref{eqn:S'slice}), and (\ref{eqn:tauhtbd}), we obtain \begin{equation*}\begin{split} \sum_{(P_i,P_j)\in W}\sum_{v\in M_K\setminus S'}r_v\log|P_i-P_j|_v\le&\sum_{(P_i,P_j)\in W} \sum_{v\in M_K^\infty\cup\mathscr{S}_d} r_v\left(\hat{\lambda}_v(P_i)+\hat{\lambda}_v(P_j)+\lambda_{\textup{crit},v}(f)+C'\right)\\&+\sum_{(P_i,P_j)\in W} \sum_{\substack{v\textup{ good}\\v\notin\mathscr{S}_d}}r_v\left(\hat{\lambda}_v(P_i)+\hat{\lambda}_v(P_j)\right)\\&+\sum_{(P_i,P_j)\in W}\sum_{v\in (M_K\setminus S)\setminus S'}r_v\left(\hat{\lambda}_v(P_i)+\hat{\lambda}_v(P_j)+\lambda_{\textup{crit},v}(f)\right)\\ \le&|W|(\tau h_{\textup{crit}}(f)+(\xi+1/\zeta) h_{\textup{crit}}(f)+C'(1+d)).\end{split}\end{equation*} By the product formula, this proves our claim (\ref{eqn:avg2}).

		We claim that when $h_{\mathrm{crit}}(f)\gg_{d,\epsilon}1$,  $\xi,\frac{1}{\zeta}\ll_{\epsilon}1$, $\kappa\ll_{d,\epsilon,\zeta}1$, and $n(n-1)\ge\frac{1}{\epsilon}$, at least $(1-\epsilon)(n(n-1))$ of the elements $(P_i,P_j)\in T^2\setminus\Delta(T^2)$ must satisfy \begin{equation}\label{eqn:avg4} \sum_{v\in S'}r_v\log|P_i-P_j|_v\le\epsilon\sum_{v\in S'}r_v\lambda_{\textup{crit},v}(f).\end{equation} To see this, suppose otherwise. Then at least $\epsilon(n(n-1))$ of these elements satisfy \begin{equation}\label{eqn:exceptional}\sum_{v\in S'}r_v\log|P_i-P_j|_v>\epsilon\sum_{v\in S'}r_v\lambda_{\textup{crit},v}(f)\end{equation} Assuming $n(n-1)\ge\frac{1}{\epsilon}$, we may choose a set of elements $U\subseteq T^2\setminus\Delta(T^2)$ satisfying (\ref{eqn:exceptional}) with $\epsilon(n(n-1))\le|U|\le 2\epsilon(n(n-1))$, so that $W=(T^2\setminus\Delta(T^2))-U$ has size at least $(1-2\epsilon)(n(n-1))$. We have by (\ref{eqn:kappabound}) that \begin{equation}\label{eqn:taubd}\begin{split}\frac{1}{|W|}\sum_{(P_i,P_j)\in W}\hat{h}_f(P_i)+\hat{h}_f(P_j)\le\frac{2\kappa(n(n-1))}{|W|}h_{\mathrm{crit}}(f)&\le\frac{1}{1-2\epsilon}(2\kappa)h_{\mathrm{crit}}(f)\\&=:\tau h_{\mathrm{crit}}(f).\end{split}\end{equation}

		By (\ref{eqn:avg2}) applied to this $W$ and the $\tau$ in (\ref{eqn:taubd}), it follows that for all $h_{\textup{crit}}(f)\gg_{d,\epsilon}1$ and $\xi,\frac{1}{\zeta}\ll_\epsilon1$, and $\kappa\ll_{d,\epsilon,\zeta}1$ that \begin{align*}\frac{1}{n(n-1)}\sum_{P_i\ne P_j\in T}\sum_{v\in S'}r_v\log|P_i-P_j|_v=&\frac{1}{n(n-1)}\Biggl(\sum_{(P_i,P_j)\in U}\sum_{v\in S'}r_v\log|P_i-P_j|_v\\&+\sum_{(P_i,P_j)\in W}\sum_{v\in S'}r_v\log|P_i-P_j|_v\Biggr)\\>&\epsilon\left(\epsilon\sum_{v\in S'}r_v\lambda_{\textup{crit},v}(f)\right)\\&+(1-\epsilon)\left(-(\xi+1/\zeta+\tau+\epsilon^2/2)h_{\textup{crit}}(f)\right)\\\ge&\epsilon^2\sum_{v\in S'}r_v\lambda_{\textup{crit},v}(f)+\frac{(1-\epsilon)(-\xi-1/\zeta-\tau-\epsilon^2/2)}{1-\frac{1}{\zeta}}\sum_{v\in S'}r_v\lambda_{\textup{crit},v}(f)\\=&\left(\epsilon^2-\frac{(\xi+1/\zeta+\tau+\epsilon^2/2)(1-\epsilon)}{1-\frac{1}{\zeta}}\right)\sum_{v\in S'}r_v\lambda_{\textup{crit},v}(f)\\>&\frac{\epsilon^2}{2}\sum_{v\in S'}r_v\lambda_{\textup{crit},v}(f),\end{align*} contradicting (\ref{eqn:avg1}). (Note that we are using (\ref{eqn:S'slice}) in obtaining the second inequality. Note also that $\frac{1}{\zeta}\ll_{\epsilon}1$ is compatible with our assumption that $\zeta>\max\left\{1,\frac{1}{\epsilon}\right\}$.) This proves the claim (\ref{eqn:avg4}).
		
		Let us denote the set of elements of $T^2\setminus\Delta(T^2)$ satisfying (\ref{eqn:avg4}) by $Z$. Recall that $S'$ and $S$ are disjoint by hypothesis. From (\ref{eqn:avg4}), (\ref{eqn:S'slice}), and (\ref{eqn:stronger}), we further obtain that for each $(P_i,P_j)\in Z$, \begin{equation}\label{eqn:avg5}\begin{split}\sum_{v\in M_K\setminus S}r_v\log|P_i-P_j|_v&=\sum_{v\in S'}r_v\log|P_i-P_j|_v+\sum_{v\in (M_K\setminus S)\setminus S'}r_v\log|P_i-P_j|_v\\&\le \epsilon\sum_{v\in S'}r_v\lambda_{\mathrm{crit},v}(f)+\sum_{v\in (M_K\setminus S)\setminus S'}r_v\log|P_i-P_j|_v\\&\le\left(\epsilon\sum_{v\in S'}r_v\lambda_{\mathrm{crit},v}(f)\right)+\frac{1}{\zeta}h_{\mathrm{crit}}(f)+\hat{h}_f(P_i)+\hat{h}_f(P_j) \\&<2\epsilon h_{\mathrm{crit}}(f)+\hat{h}_f(P_i)+\hat{h}_f(P_j)\end{split}\end{equation} as we had assumed that $\frac{1}{\zeta}<\epsilon$. The product formula applied to (\ref{eqn:avg5}) then leads us to conclude that for each $(P_i,P_j)\in Z$, \begin{equation}\label{eqn:adgood1}\begin{split}\sum_{v\in S}r_v\log|(P_i-P_j)|_v>-\left(2\epsilon h_{\mathrm{crit}}(f)+\hat{h}_f(P_i)+\hat{h}_f(P_j)\right).\end{split}\end{equation} The right-hand side of this inequality is in turn bounded below by $-3\epsilon h_{\textup{crit}}(f)$ for at least $(1-\epsilon)|Z|\ge (1-\epsilon)^2(n(n-1))$ elements $(P_i,P_j)\in Z$, provided $\kappa\ll_{\epsilon}1$. Without loss make this assumption on $\kappa$, and call the set of such elements $Z'$, so that \begin{equation}\label{eqn:adgood1'}\begin{split}\sum_{v\in S}r_v\log|(P_i-P_j)|_v\ge-3\epsilon h_{\textup{crit}}(f)\end{split}\end{equation} for all $(P_i,P_j)\in Z'$.

		Now suppose that $(P_i,P_j)\in Z'$, and that $\kappa'>0$ satisfies \begin{equation}\label{eqn:kappaprime}\hat{h}_f(P_i),\hat{h}_f(P_j)\le\kappa' h_{\textup{crit}}(f).\end{equation} By (\ref{eqn:xichoice}) and (\ref{eqn:C'}), we have \begin{equation}\label{eqn:bound1}\begin{split}\sum_{v\in S}r_v\log\max\{1,|P_i-P_j|_v\}&\le2\kappa' h_{\textup{crit}}(f)+(d+1)C'+\sum_{v\in M_K^\infty\cup\mathscr{S}_d}r_v\lambda_{\textup{crit},v}(f)\\&\le(2\kappa'+\xi)h_{\textup{crit}}(f)+(d+1)C'.\end{split}\end{equation} Subtracting (\ref{eqn:adgood1'}) from (\ref{eqn:bound1}) gives \begin{equation}\label{eqn:bound2}\begin{split}\sum_{v\in S}r_v\left(-\log\min\{1,|P_i-P_j|_v\}\right)&\le\left(2\kappa'+\xi+3\epsilon\right)h_{\textup{crit}}(f)+(d+1)C'.\end{split}\end{equation} Inequalities (\ref{eqn:bound1}) and (\ref{eqn:bound2}) taken together give \begin{equation}\label{eqn:combination}\begin{split}\sum_{v\in S} r_v\left|\log|(P_i-P_j)^{d-1}|_v\right|&\le (d-1)(4\kappa'+2\xi+3\epsilon)h_{\textup{crit}}(f)+2(d-1)(d+1)C'.\end{split}\end{equation} Since by (\ref{eqn:kappabound}) we have \[\frac{1}{n(n-1)}\sum_{(P_i,P_j)\in T^2\setminus\Delta(T^2)}\hat{h}_f(P_i)+\hat{h}_f(P_j)\le 2\kappa h_{\textup{crit}}(f),\] one obtains \[\frac{1}{|Z'|}\sum_{(P_i,P_j)\in Z'}\hat{h}_f(P_i)+\hat{h}_f(P_j)\le\frac{2}{(1-\epsilon)^2}\kappa h_{\textup{crit}}(f).\] It follows that at least $(1-\epsilon)|Z'|\ge(1-\epsilon)^3n(n-1)$ of the elements $(P_i,P_j)\in Z'$ satisfy \[\hat{h}_f(P_i),\hat{h}_f(P_j)\le\frac{2}{\epsilon(1-\epsilon)^2}\kappa h_{\textup{crit}}(f).\] Letting $\kappa'=\frac{2}{\epsilon(1-\epsilon)^2}\kappa$ in inequalities (\ref{eqn:kappaprime})--(\ref{eqn:combination}), we see that (\ref{eqn:combination}) implies that for $\kappa\ll_{d,\epsilon,[K:F]}1$ and $n\gg_{d,\epsilon,[K:F]}1$, at least $(1-\epsilon)^3n(n-1)$ elements of $T^2\setminus\Delta(T^2)$ have the property that \[\sum_{v\in S} r_v\left|\log|(P_i-P_j)^{d-1}|_v\right|\le 3d\epsilon h_{\textup{crit}}(f).\] Adjusting the initial choice of $\epsilon$ accordingly completes the proof.\end{proof}

	\section{Proof of Theorems \ref{thm:UBCpolys} and \ref{thm:dynLangpolys}}\label{section:proof}
	
	In this section, we prove Theorems \ref{thm:UBCpolys} and \ref{thm:dynLangpolys}. These theorems follow from the following two statements. 
	
	\begin{thm}\label{thm:combined} Assume the $abcd$-conjecture (Conjecture \ref{conj:abcd}) for $K$, and let $d\ge2$. If $K$ is a function field, then assume that $\operatorname{char}(K)=0$. There is a $\kappa=\kappa(d,K)>0$ and a $B=B(d,K)$ such that if $f(z)\in K[z]$ is of degree $d$, and if $T\subseteq K$ is a finite set satisfying \[\frac{1}{|T|}\sum_{P_i\in T}\hat{h}_f(P_i)\le\kappa \max\{1,h_{\textup{crit}}(f)\},\] then either $K$ is a function field and $f$ is isotrivial, or $|T|\le B$.\end{thm}
	
	\begin{lem}\label{lem:isotrivLang} \sloppy Let $K/k(t)$ be a function field, and let $d\ge2$. There is a ${\kappa=\kappa(d,[K:k(t)])>0}$ with the following property. If $f\in K[z]$ of degree $d$ is conjugate by $\mu\in\textup{PGL}_2(\overline{K})$ to $g\in k[z]$, then for all $P\in K$, either $\hat{h}_f(P)=0$, or \[\hat{h}_f(P)\ge\kappa.\]\end{lem}
	
	\begin{proof} Let $f\in K[z]$ of degree $d$ be isotrivial, i.e., assume that $\mu f\mu^{-1}\in k[z]$ for some $\mu\in\textup{PGL}_2(\overline{K})$. By \cite[Proposition 6.1]{Benedetto1}, isotriviality is equivalent to everywhere potential good reduction, so the filled Julia set $\mathcal{K}_v$ of $f$ must be a disk at all places. We have seen that the capacity of such a disk of radius $r$ is equal to $r$. We also know that this capacity equals $|a_d|_v^{-1/(d-1)}$, where $a_d$ is the lead coefficient of $f$ \cite[Theorem 1.2]{DeMarcoRumely}. Conjugating $f$ by an affine linear transformation $\mu\in\textup{PGL}_2(L)$ for a suitable extension $L/K$, in such a way that $g=\mu f\mu^{-1}\in k[z]$ is monic and defined over the field of constants $k$, the filled Julia set of $g$ is simply the closed unit disk about $0$ at all places $w\in M_L$ (since then both this disk and its complement are totally invariant under $g$).
		
		Suppose $P\in K$. As $\mu$ corresponds to a scaling by $a_d^{1/(d-1)}$ followed by a translation taking the filled Julia set to the closed unit disk about $0$, we may assume that $[L:K]\le d(d-1)$. Thus for any $v\in M_K$ and for any extension $w$ of $v$ to $L$, \begin{equation}\label{eqn:pointvaluation} |\mu(P)|_w\in|K^\ast|_v^{1/(d(d-1))!}\cup\{0\}.\end{equation} Suppose $P$ satisfies $0<\hat{h}_g(\mu(P))\le\kappa$ (which is equivalent to $0<\hat{h}_f(P)\le\kappa$). Then for some $w\in M_L$, we have $0<u_w\hat{\lambda}_w(\mu(P))\le\kappa$, where \[u_w:=\frac{[L_w:k(t)_v]}{[L:k(t)]}=\frac{[L_w:K_v][K_v:k(t)_v]}{[L:K][K:k(t)]}\ge\frac{1}{d(d-1)}\frac{[K_v:k(t)_v]}{[K:k(t)]}\] and $\hat{\lambda}_w$ is the $w$-adic local canonical height with respect to $g$. Since the $w$-adic filled Julia set of $g$ is the closed unit disk about $0$, the local canonical height $\hat{\lambda}_w$ satisfies $\hat{\lambda}_w(\mu(P))=\log\max\{1,|\mu(P)|_w\}$, from which we obtain $0<u_w\log|\mu(P)|_w\le\kappa$. But if $\kappa\ll_{d,[K:k(t)]}1$, this contradicts (\ref{eqn:pointvaluation}). The desired result follows.\end{proof}
	
	For the purposes of the following proposition, we remind the reader that for two Berkovich disks $\mathcal{B}_1$ and $\mathcal{B}_2$, \[\delta_v(\mathcal{B}_1,\mathcal{B}_2):=\sup_{x\in\mathcal{B}_1,y\in\mathcal{B}_2}\{\delta_v(x,y)\}.\] 
	
	\begin{prop}\label{prop:quad} Let $0<\epsilon<1$ and let $d\ge2$. Let $f(z)\in K[z]$ be a degree $d$ polynomial, and suppose $v\in M_K^0\setminus\mathscr{S}_d$ is a place of bad reduction for $f$. Then if $T\subseteq K$ is $\frac{1}{2}$-equidistributed at $v$, and $k\ge k_0:=\left\lceil \frac{\log\epsilon}{C(d)} \right\rceil$, where $C(d)=\log\left(1-\left(\frac{1}{2d(d-1)}\right)^4\right)$, then at least $(1-\epsilon)|T|^{4k}$ elements of the form \[(a_1,b_1,c_1,d_1,\dots,a_k,b_k,c_k,d_k)\in T^{4k}\] have the property that for some $1\le i\le k$, there are disk components $\mathcal{B}_{1,j}$ and $\mathcal{B}_{1,l}$ of $\mathcal{E}_1$ with $\log\delta_v(\mathcal{B}_{1,j},\mathcal{B}_{1,l})=g_v$ and \[a_i,b_i\in \mathcal{B}_{1,j} \text{ and } c_i,d_i\in\mathcal{B}_{1,l}.\] 
	\end{prop}
	
	\begin{proof}Let $0<\epsilon<1$, and let $T$ be $\frac{1}{2}$-equidistributed at $v$. Let $A$ and $B$ be the wings of a wing decomposition of $\mathcal{E}_1$. Since $A$ contains at most $d-1$ disk components, it follows from the definition of $\frac{1}{2}$-equidistribution and the pigeonhole principle that there is a disk component $\mathcal{B}_{1,j}$ of $A$ such that at least $|T|(1/(2d(d-1))$ elements of $T$ lie in $\mathcal{B}_{1,j}$, and hence at least $|T^2|(1/(2d(d-1)))^2$ elements of $T^2$ lie in $\mathcal{B}_{1,j}$. Similarly, there is a disk component $\mathcal{B}_{1,l}$ of $B$ such that at least $|T|(1/(2d(d-1))$ elements of $T$ lie in $\mathcal{B}_{1,l}$, and at least $|T^2|(1/(2d(d-1))^2$ elements of $T^2$ lie in $\mathcal{B}_{1,l}$. Thus, writing $t(d):=1/(2d(d-1))^4$, at least $t(d)|T^4|$ elements $(a_i,b_i,c_i,d_i)\in T^4$ satisfy $a_i,b_i\in\mathcal{B}_{1,j}$ and $c_i,d_i\in\mathcal{B}_{1,l}$. From this it follows that at least \[(1-(1-t(d))^k)|T^{4k}|\] choices of \[(a_1,b_1,c_1,d_1,\dots,a_k,b_k,c_k,d_k)\in T^{4k}\] have the property that for some $1\le i\le k$, \[a_i,b_i\in\mathcal{B}_{1,j} \text{ and } c_i,d_i\in\mathcal{B}_{1,l}.\] Taking $C(d)=\log\left(1-t(d)\right)$ yields $(1-(1-t(d))^k)\ge 1-\epsilon$ so long as $k\ge\frac{\log\epsilon}{C(d)}$.\end{proof}

	\begin{cor}\label{cor:quad}Let $0<\epsilon<1$, and let $d\ge2$. There are constants $\kappa=\kappa(d,\epsilon)>0$, $N=N(d,\epsilon)$, and $k_0=k_0(d,\epsilon)$ with the following property. Let $f(z)\in K[z]$ be a degree $d$ polynomial, and suppose $T\subseteq K$ is a finite set with $|T|\ge N$ and \[\frac{1}{|T|}\sum_{P_i\in T}\hat{h}_f(P_i)\le\kappa\max\{1,h_{\textup{crit}}(f)\}.\] For each $v\in M_K^0\setminus\mathscr{S}_d$, let $\mathcal{E}_1$ be as in (\ref{eqn:Emdef}), and suppose $k\ge k_0$. Then at least $(1-\epsilon)|T|^{4k}$ elements of the form \[(a_1,b_1,c_1,d_1,\dots,a_k,b_k,c_k,d_k)\in T^{4k}\] have the property that there is a $(1-\epsilon)$-slice $S$ of bad places $v\in M_K^0\setminus\mathscr{S}_d$ such that for any $v\in S$, there is some $1\le i\le k$ such that \[a_i,b_i\in\mathcal{B}_{1,j} \text{ and }c_i,d_i\in\mathcal{B}_{1,l},\] for some disk components $\mathcal{B}_{1,j}$, $\mathcal{B}_{1,l}$ of $\mathcal{E}_1$ satisfying $\log\delta_v(\mathcal{B}_{1,j},\mathcal{B}_{1,l})=g_v$. Furthermore, we may let $k_0=\left\lceil \frac{\log(\epsilon^2/2)}{C(d)} \right\rceil$, where $C(d)$ is as in Proposition \ref{prop:quad}.\end{cor} 
	
	\begin{proof} Let $0<\epsilon<1$, let $0<\epsilon'<\frac{1}{2}$ be such that $\epsilon'<\epsilon^2$, and let $\delta<1$ be such that $\delta(1-\epsilon')>1-\epsilon^2$. For $v\in M_K^0\setminus\mathscr{S}_d$ a place of bad reduction, $k\ge k_0=k_0(d,\epsilon')$ as in Proposition \ref{prop:quad}, and \[\vec{x}=(a_1,b_1,c_1,d_1,\dots,a_k,b_k,c_k,d_k)\in T^{4k},\] let $\chi_v(\vec{x})=1$ if for some $1\le i\le k$ we have $a_i,b_i\in\mathcal{B}_{1,j}$ and $c_i,d_i\in\mathcal{B}_{1,l}$ for disk components $\mathcal{B}_{1,j}$, $\mathcal{B}_{1,l}$ such that $\log\delta_v(\mathcal{B}_{1,j},\mathcal{B}_{1,l})=\lambda_{\textup{crit},v}(f)$, and let $\chi_v(\vec{x})=0$ otherwise. Proposition \ref{prop:quad} says that if $T$ is $\epsilon'$-equidistributed at $v$, then \[\frac{1}{|T|^{4k}}\sum_{\vec{x}\in T^{4k}}\chi_v(\vec{x})\ge(1-\epsilon').\] From Theorem \ref{thm:globaleq}, we deduce that if $|T|\gg_{\epsilon',d,\delta}1$ and $\kappa\ll_{\epsilon',d,\delta}1$, then $T$ is $\epsilon'$-equidistributed at a $\delta$-slice of bad places $v\in M_K^0\setminus\mathscr{S}_d$, and hence \begin{equation}\label{eqn:chitotal}\frac{1}{|T|^{4k}}\sum_{v\in M_K^0\setminus\mathscr{S}_d}\sum_{\vec{x}\in T^{4k}}r_v\chi_v(\vec{x})\lambda_{\textup{crit},v}(f)\ge\delta(1-\epsilon')\sum_{v\in M_K^0\setminus\mathscr{S}_d}r_v\lambda_{\textup{crit},v}(f).\end{equation}
		
		Now suppose the set $\mathcal{T}$ of tuples in $T^{4k}$ such that there does \emph{not} exist a $(1-\epsilon)$-slice of places as stipulated in the statement of Corollary \ref{cor:quad} is nonempty. Then \[\sum_{\vec{x}\in\mathcal{T}}\sum_{v\in M_K^0\setminus\mathscr{S}_d}r_v\chi_v(\vec{x})\lambda_{\textup{crit},v}(f)<|\mathcal{T}|(1-\epsilon)\sum_{v\in M_K^0\setminus\mathscr{S}_d}r_v\lambda_{\textup{crit},v}(f),\] and trivially, as $\chi_v(\vec{x})\le 1$, \[\sum_{\vec{x}\in T^{4k}\setminus\mathcal{T}}\sum_{v\in M_K^0\setminus\mathscr{S}_d}r_v\chi_v(\vec{x})\lambda_{\textup{crit},v}(f)\le |T^{4k}\setminus\mathcal{T}|\sum_{v\in M_K^0\setminus\mathscr{S}_d}r_v\lambda_{\textup{crit},v}(f).\] Summing these last two inequalities yields \begin{equation}\label{eqn:tupledistribution} \sum_{\vec{x}\in T^{4k}}\sum_{v\in M_K^0\setminus\mathscr{S}_d}r_v\chi_v(\vec{x})\lambda_{\textup{crit},v}(f)<\left(|\mathcal{T}|(1-\epsilon)+|T^{4k}\setminus\mathcal{T}|\right)\sum_{v\in M_K^0\setminus\mathscr{S}_d}r_v\lambda_{\textup{crit},v}(f).\end{equation} Write $k_1=\frac{|\mathcal{T}|}{|T|^{4k}}$ and $k_2=\frac{|T^{4k}\setminus\mathcal{T}|}{|T|^{4k}}$. Then dividing (\ref{eqn:tupledistribution}) by $|T|^{4k}$ gives \begin{equation}\label{eqn:linearcombination}\frac{1}{|T|^{4k}}\sum_{\vec{x}\in T^{4k}}\sum_{v\in M_K^0\setminus\mathscr{S}_d}r_v\chi_v(\vec{x})\lambda_{\textup{crit},v}(f)<\left(k_1(1-\epsilon)+k_2\cdot 1\right)\sum_{v\in M_K^0\setminus\mathscr{S}_d}r_v\lambda_{\textup{crit},v}(f).\end{equation} Suppose that less than $(1-\epsilon)|T|^{4k}$ elements of $T^{4k}$ are in $T^{4k}\setminus\mathcal{T}$, i.e., that $k_1\ge \epsilon$ and $k_2< 1-\epsilon$. Since the coefficient outside the summation on the right-hand side of (\ref{eqn:linearcombination}) is an affine-linear combination of $1-\epsilon$ and $1$, and hence increases monotonically as $k_2$ increases, we then have \begin{equation*}\begin{split} \frac{1}{|T|^{4k}}\sum_{\vec{x}\in T^{4k}}\sum_{v\in M_K^0\setminus\mathscr{S}_d}r_v\chi_v(\vec{x})\lambda_{\textup{crit},v}(f)&<\left(\epsilon(1-\epsilon)+(1-\epsilon)\cdot1\right)\sum_{v\in M_K^0\setminus\mathscr{S}_d}r_v\lambda_{\textup{crit},v}(f)\\&=(1-\epsilon^2)\sum_{v\in M_K^0\setminus\mathscr{S}_d}r_v\lambda_{\textup{crit},v}(f).\end{split}\end{equation*} As we had assumed that $\delta(1-\epsilon')>1-\epsilon^2$, this forces a contradiction of (\ref{eqn:chitotal}) whenever $|T|\gg_{\epsilon',d,\delta}1$ and $\kappa\ll_{\epsilon',d,\delta}1$ so as to make (\ref{eqn:chitotal}) hold. We conclude that under this assumption, $|T^{4k}\setminus\mathcal{T}|\ge(1-\epsilon)|T|^{4k}$, as needed. The corollary follows upon observing that we may choose for example $\epsilon'=\epsilon^2/2$ in the foregoing, so that $k_0=\left\lceil \frac{\log(\epsilon^2/2)}{C(d)} \right\rceil$ depends only on $d$ and $\epsilon$.\end{proof}

	\begin{proof}[Proof of Theorem \ref{thm:combined}] Let $0<\epsilon<1$, suppose without loss that $f$ is non-isotrivial if $K$ is a function field, and suppose $T\subseteq K$ is a finite set and \begin{equation}\label{startassump} \frac{1}{|T|}\sum_{P_i\in T}\hat{h}_f(P_i)\le\kappa\max\{1,h_{\textup{crit}}(f)\}.\end{equation}Let $k=k_0:=\left\lceil \frac{\log(\epsilon^2/2)}{C(d)} \right\rceil$ be as in Corollary \ref{cor:quad}. Let $Z_1,\dots,Z_{2k+1}$ be the standard homogeneous coordinates on $\mathbb{P}^{2k}$, and let \[\mathcal{H}:Z_1+\dots+Z_{2k+1}=0.\] Suppose the corresponding set $\Zcal=\Zcal(K,\epsilon,2k+1)\subsetneq \mathcal{H}$ stipulated by Conjecture \ref{conj:abcd} is contained in the hypersurface defined by $g(Z_1,\dots,Z_{2k+1})=0$ for some $g\in\overline{K}[Z_1,\dots,Z_{2k+1}]$. If $\mathcal{Z}$ is empty (which can only happen when $2k+1=3$ by \cite[p.~6]{Vojta}), then the discussion that follows, resulting in (\ref{g2conclusion}), is unnecessary anyway as there is then no exceptional set to avoid. We thus assume that $\mathcal{Z}$ is nonempty. By symmetry, we may assume that $g$ is non-constant in each of the variables $Z_1,\dots,Z_{2k+1}$; indeed, if $g$ does not have this property, then we may replace $g$ by the product of all polynomials in the orbit of $g$ under the permutation action on $Z_1,\dots,Z_{2k+1}$. Since on $\mathcal{Z}$ we have $Z_{2k+1}=-\sum_{i=1}^{2k}Z_{2k}$, there is a $g_1\in\overline{K}[Z_1,\dots,Z_{2k}]$ such that $g(Z_1,\dots,Z_{2k+1})=g_1(Z_1,\dots,Z_{2k})$ on $\mathcal{Z}$. For each $1\le j\le k$, write $Z_{2j-1}=X_j$ and $Z_{2j}=M_j-X_j$. Substituting these expressions for the $Z_j$, we have \[g_1(Z_1,\dots,Z_{2k})=g_2(X_1,M_1,\dots,X_k,M_k),\] where $g_2$ is a polynomial over $\overline{K}$ non-constant in each of the variables $M_1,\dots,M_k$, and non-constant in the variables in some (a priori possibly empty) subset \begin{equation}\label{eqn:specialsubset}\{X_{i_1},\dots,X_{i_s}\}\subseteq\{X_1,\dots,X_k\}.\end{equation} We note that the case where this subset is empty immediately allows us to choose $m_1,\dots,m_k$ such that $g_2(x_1,m_1,\dots,x_k,m_k)\ne 0$ for any $x_1,\dots,x_k$. As producing $m_1,\dots,m_k$ (depending only on $d$) and suitable $x_1,\dots,x_k$ (derived from $f$) such that $g_2(x_1,m_1,\dots,x_k,m_k)\ne0$ is the entire goal of all subsequent discussion involving $\mathcal{Z}$, we may therefore assume that $s\ge1$. All but finitely many choices of $M_1=m_1\in\mathbb{Z}_+$ yield $g_2(X_1,m_1,\dots,X_k,M_k)$ non-constant in each of the variables $X_{i_1},X_{i_2},\dots,X_{i_s}$,$M_2,\dots,M_k$. Choose such an $m_1$, and make successive choices of $m_2,\dots,m_k\in\mathbb{Z}_+$ such that \[g_2(X_1,m_1,\dots,X_k,m_k)\] is non-constant in each of the variables $X_{i_1},\dots,X_{i_s}$. These values of $m_1,\dots,m_k$ are fixed in all that follows. 
		
		We next make three straightforward yet verbally cumbersome observations that we will then combine. (The second ``observation'' simply applies Corollary \ref{cor:quad} to our set $T$, and is written here for convenience.)
		\newline 
		
		\textbf{Observation 1}: Reindex the $X_j$ so that (\ref{eqn:specialsubset}) satisfies $i_t=t$ for all $1\le t\le s$. For each $j\le k-1$ and each successive substitution (into $g_2$) of the coordinates $x_k=X_k,x_{k-1}=X_{k-1},\dots,x_{k-j+1}=X_{k-j+1}$ such that $g_2$ is nonconstant in all of the indeterminates $X_i$ with $1\le i\le\min\{s,k-j\}$, there are at most $C=C(\textup{deg}(g_2),k)=C(d)$ choices of $x_{k-j}=X_{k-j}$ such that \[g_2\left(X_1,\dots,X_{k-j-1},x_{k-j},\dots,x_k\right)\] is constant in one or more of the variables $X_1,\dots,X_{\min\{s,k-j-1\}}$. Here $\textup{deg}(g_2)$ is the total degree of $g_2$ as a polynomial in the $k$ variables $X_1,\ldots,X_k$.
		
		\textbf{Observation 2}: For each $v\in M_K^0\setminus\mathscr{S}_d$, let $\mathcal{E}_1$ be as in (\ref{eqn:Emdef}). As we have assumed that $k\ge k_0=k_0(d,\epsilon)$ for $k_0$ as in Corollary \ref{cor:quad}, it follows that if $|T|\gg_{d,\epsilon}1$, and $\kappa\ll_{d,\epsilon}1$, then at least $(1-\epsilon)|T|^{4k}$ choices of \begin{equation}\label{eqn:4ktuple}\vec{x}=(a_1,b_1,c_1,d_1,\dots,a_k,b_k,c_k,d_k)\in T^{4k}\end{equation} have the property that there is a $(1-\epsilon)$-slice $S$ of bad places $v\in M_K^0\backslash\mathscr{S}_d$ such that for each $v\in S$, there is some $1\le i\le k$ as well as disk components $\mathcal{B}_{1,j}$, $\mathcal{B}_{1,l}$ of $\mathcal{E}_1$ such that \begin{equation}\label{eqn:Sslice}a_i,b_i\in\mathcal{B}_{1,j} \text{ and }c_i,d_i\in\mathcal{B}_{1,l}\end{equation} and $\log\delta_v(\mathcal{B}_{1,j},\mathcal{B}_{1,l})=g_v$. 
		
		\textbf{Aside to Observation 2}: We further note that for such points $a_i,b_i\in\mathcal{B}_{1,j} \text{ and }c_i,d_i\in\mathcal{B}_{1,l}$, \begin{equation}\label{eqn:quotientbd1}\log\left|\frac{a_i-c_i}{c_i-d_i}\right|_v\ge\lambda_{\textup{crit},v}(f);\end{equation} indeed, both sides are invariant under conjugation and (\ref{eqn:quotientbd1}) clearly holds in the case where $f$ is monic, as then we have \begin{equation*}\label{eqn:diamlb} \log\text{diam}(\mathcal{B}_{1,l})\le0\end{equation*} by \cite[Proposition 4.3]{Looper:mincanht} and $\log|a_i-c_i|_v=g_v=\lambda_{\textup{crit},v}(f)$ by (\ref{eqn:splittingradiusmonic}). Since $N_v = r_v \log|\pi_v^{-1}|_v$ where $\pi_v$ is a uniformizer at $v$, one further has \begin{equation}\label{eqn:quotientbd2}r_v\log\left|\frac{a_i-c_i}{c_i-d_i}\right|_v\ge N_v\end{equation} by the $K$-rationality of $a_i,c_i,d_i$. Inequalities analogous to those in (\ref{eqn:quotientbd1}) and (\ref{eqn:quotientbd2}) also apply to $\left|\frac{d_i-b_i}{b_i-a_i}\right|_v$. It follows that \begin{equation}\label{eqn:finalquotientbd}r_v\log\left|\frac{(a_i-c_i)(d_i-b_i)}{(c_i-d_i)(b_i-a_i)}\right|_v-N_v\ge r_v\log\left|\frac{d_i-b_i}{b_i-a_i}\right|_v\ge r_v\lambda_{\textup{crit},v}(f)\end{equation} for any $a_i,b_i,c_i,d_i$ satisfying (\ref{eqn:Sslice}). In the reverse direction, (\ref{eqn:unif}) below gives \begin{equation}\label{eqn:Nvlambda} N_v\le d(d-1)\lambda_{\textup{crit},v}(f)\end{equation} for all $v\in S$. We will later apply (\ref{eqn:finalquotientbd}) and (\ref{eqn:Nvlambda}) in order to obtain a key immediate feeder into our final inequality (\ref{eqn:lastline}) (namely (\ref{eqn:SnotSprime})). 
		
		For $\vec{a}=(a_1,\dots,a_k)\in T^k$, and similarly for $\vec{b},\vec{c},\vec{d}$, we write $(\vec{a},\vec{b},\vec{c},\vec{d})$ for the element \[(\vec{a},\vec{b},\vec{c},\vec{d})=(a_1,b_1,c_1,d_1,\dots,a_k,b_k,c_k,d_k)\in T^{4k}.\] It follows from Proposition \ref{prop:adgoodbars} that if $|T|\gg_{\epsilon,d}1$ and $\kappa\ll_{\epsilon,d}1$, then at least $\left(1-\epsilon\right)|T^{4k}|$ elements of $(\vec{a},\vec{b},\vec{c},\vec{d})\in T^{4k}$ have the property that for all $1\le i\le k$, \begin{equation}\label{eqn:egood} a_i-b_i,a_i-c_i,a_i-d_i,b_i-c_i,b_i-d_i,  \text{ and }  c_i-d_i \textup{ are }\frac{\epsilon}{8}\textup{-adelically good}.\end{equation} Observation $2$ on the other hand says that if $n\gg_{\epsilon,d}1$ and $\kappa\ll_{\epsilon,d}1$, then at least $\left(1-\epsilon\right)|T^{4k}|$ elements of $T^{4k}$ have the property that (\ref{eqn:Sslice}) holds for some $(1-\epsilon)$-slice of bad places $v\in M_K^0\setminus\mathscr{S}_d$ (depending on $(\vec{a},\vec{b},\vec{c},\vec{d})$). Denoting the set of tuples satisfying (\ref{eqn:egood}) by $T_1$ and the latter set of tuples by $T_2$, we thus have that $|T_1\cap T_2|\ge(1-2\epsilon)|T^{4k}|$ for $|T|\gg_{\epsilon,d}1$ and $\kappa\ll_{\epsilon,d}1$. Let $W:=T_1\cap T_2$. For $\vec{a}=(a_1,\dots,a_k)$, $\vec{b}=(b_1,\dots,b_k)$, and $\vec{c}=(c_1,\dots,c_k)$, write  \[W_{(\vec{a},\vec{b},\vec{c})}:=\{(a_1,b_1,c_1,d_1,\dots,a_k,b_k,c_k,d_k)\in W\}\] and partition $W$ as \begin{equation}\label{eqn:partition}W=\bigsqcup_{\substack{(a_1,b_1,c_1,\dots,a_k,b_k,c_k)\in T^{3k}\\W_{(\vec{a},\vec{b},\vec{c})}\ne\emptyset}}W_{(\vec{a},\vec{b},\vec{c})}.\end{equation}

		\textbf{Observation 3}: For $|T|\gg_{\epsilon,d}1$ and $\kappa\ll_{\epsilon,d}1$, at least $(1-\sqrt{2\epsilon})|T^{3k}|$ elements \[(a_1,b_1,c_1,\dots,a_k,b_k,c_k)\in T^{3k}\] each have the property that at least $(1-\sqrt{2\epsilon})|T^k|$ elements $\vec{d}=(d_1,\dots,d_k)\in T^k$ satisfy \begin{equation}\label{eqn:belonging}(a_1,b_1,c_1,d_1,\dots,a_k,b_k,c_k,d_k)\in W_{(\vec{a},\vec{b},\vec{c})}.\end{equation} For otherwise, by (\ref{eqn:partition}), more than $\sqrt{2\epsilon}|T^{3k}|$ elements $(a_1,b_1,c_1,\dots,a_k,b_k,c_k)\in T^{3k}$ each have more than $\sqrt{2\epsilon}|T^{k}|$ elements $\vec{d}\in T^k$ with $(\vec{a},\vec{b},\vec{c},\vec{d})\notin W$, i.e., there are more than $2\epsilon|T^{4k}|$ elements of $T^{4k}$ that are not in $W$, contradicting our very first statement about $W=T_1\cap T_2$.

		Let $(a_1,b_1,c_1,\dots,a_k,b_k,c_k)\in T^{3k}$ be one of the at least $\left(1-\sqrt{2\epsilon}\right)|T^{3k}|$ resulting elements such that at least $\left(1-\sqrt{2\epsilon}\right)|T^k|$ elements $\vec{d}=(d_1,\dots,d_k)\in T^k$ satisfy (\ref{eqn:belonging}). %%Suppose $(a_1,b_1,c_1,\dots,a_k,b_k,c_k)\in W$ is such that at least $(1-\epsilon)|T^k|$ elements $\vec{d}=(d_1,\dots,d_k)\in T^k$ satisfy (\ref{eqn:belonging}).%%
		Note that since for all $1\le i\le k$, the elements $a_i-c_i$, $b_i-c_i$, and $b_i-a_i$ are $\frac{\epsilon}{8}$-adelically good, we have in particular that these elements are nonzero, and so  \[\psi_{(a_i,b_i,c_i)}(y):=m_i\frac{(a_i-c_i)(y-b_i)}{(c_i-y)(b_i-a_i)}\in K(y)\] is a fractional linear transformation, hence a bijection of $\mathbb{P}_K^1$, for all $1\le i\le k$. Applying Observation 1, we thus see that when $|T|\gg_{\epsilon,d,k=k(d)}1$ and $\kappa\ll_{\epsilon,d}1$, there are at least $\left(1-\sqrt{2\epsilon}\right)|T|-C$ elements $d_1\in T$ such that: \begin{enumerate}[label=(\roman*)] \item \[(a_1,b_1,c_1,d_1,\dots,a_k,b_k,c_k,d_k)\in W_{(\vec{a},\vec{b},\vec{c})}\] for some choices of $d_2,\dots,d_k$, and \item If $x_k=m_k(a_k-c_k)(d_k-b_k)/[(c_k-d_k)(b_k-a_k)]$, then $g_2(X_1,m_1,X_2,m_2,\dots,x_k,m_k)$ is nonconstant in all of the indeterminates $X_i$ with $1\le i\le\min\{s,k-1\}$.\end{enumerate} Continuing in this manner, it follows that when $|T|\gg_{\epsilon,d}1$ and $\kappa\ll_{\epsilon,d}1$, at least \begin{equation*}\begin{split}\left(\left(1-\sqrt{2\epsilon}\right)|T|-C\right)^k&\ge\left(1-2\sqrt{2\epsilon}\right)^k|T|^k\end{split}\end{equation*} elements $\vec{d}\in T^k$ satisfy: \begin{enumerate}[label=(\roman*), resume] \item\label{item:iii} $(a_1,b_1,c_1,d_1,\dots,a_k,b_k,c_k,d_k)\in W_{(\vec{a},\vec{b},\vec{c})}$, and \item\label{item:iv} If \begin{equation}\label{eqn:x_i}x_i=m_i\frac{(a_i-c_i)(d_i-b_i)}{(c_i-d_i)(b_i-a_i)}\hspace{10mm}\forall1\le i\le k,\end{equation} then \begin{equation}\label{g2conclusion} g_2(x_1,m_1,\dots,x_k,m_k)\ne 0.\end{equation} \end{enumerate} (Note that \ref{item:iii} ensures that $a_i-c_i,d_i-b_i,c_i-d_i,b_i-a_i$ are all nonzero, so that $x_i$ is nonzero and finite.) Altogether, we conclude that for $|T|\gg_{\epsilon,d}1$ and $\kappa\ll_{\epsilon,d}1$, at least \[\left(1-2\sqrt{2\epsilon}\right)^k|T|^k\left(1-\sqrt{2\epsilon}\right)|T^{3k}|>(1-2\sqrt{2\epsilon})^{k+1}|T^{4k}|\] elements of $W$ satisfy \ref{item:iv}.
		
		Despite having shown that the number of elements of $W\subseteq T^4$ satisfying \ref{item:iv} may be assumed to be at least $(1-2\sqrt{2\epsilon})^{k+1}|T^{4k}|$ provided $|T|\gg_{\epsilon,d}1$ and $\kappa\ll_{\epsilon,d}1$, the rest of the proof only makes use of \emph{one} such element to ultimately derive a contradiction of Conjecture \ref{conj:abcd} if $h_{\textup{crit}}(f)$ is sufficiently large. Let $\vec{x}=(a_1,b_1,c_1,d_1,\dots,a_k,b_k,c_k,d_k)$ be an element of $W$ satisfying \ref{item:iv}, and let $x_i$ be as in (\ref{eqn:x_i}) for all $1\le i\le k$. The Pl\"ucker identity reads as \begin{equation*} m_i\frac{(a_i-c_i)(d_i-b_i)}{(c_i-d_i)(b_i-a_i)}-m_i=m_i\frac{(d_i-a_i)(c_i-b_i)}{(d_i-c_i)(b_i-a_i)}\end{equation*} for all $1\le i\le k$. We have that for $S_0$ the union of the set of places of good reduction for $f$ and $\mathscr{S}_d\cup M_K^\infty$, \[A_{i,1}:=\frac{x_i}{m_i}=\frac{(a_i-c_i)(d_i-b_i)}{(c_i-d_i)(b_i-a_i)}\hspace{3mm}\textup{ and }\hspace{3mm}A_{i,2}:=\frac{m_i-x_i}{m_i}=\frac{(d_i-a_i)(c_i-b_i)}{(d_i-c_i)(b_i-a_i)}\] satisfy \begin{equation*}(d-1)\sum_{v\in S_0}r_v|\log |A_{i,j}|_v|\le\frac{\epsilon}{2}h_{\textup{crit}}(f)\text{ for }j=1,2 \textup{ and }1\le i\le k,\end{equation*} since each of the pairwise differences appearing in parentheses has been assumed to be $\frac{\epsilon}{8}$-adelically good. (Note that the role of the lead coefficient appearing in the definition of $\epsilon$-goodness cancels in the numerator and denominator of each $A_{i,j}$.) Therefore for $h_{\textup{crit}}(f)\gg_{\epsilon,m_i=m_i(d)}1$, we have that \begin{equation}\label{eqn:actualgoodness}(d-1)\sum_{v\in S_0}r_v|\log|m_iA_{i,j}|_v|\le\epsilon h_{\textup{crit}}(f)\end{equation} for $j=1,2$ and for all $1\le i\le k$. We assume in what follows that $h_{\textup{crit}}(f)$ is sufficiently large for this to hold; as will be reiterated at the end of the proof, the case where $h_{\textup{crit}}(f)$ is bounded is addressed by Proposition \ref{prop:Mxi}.  
		
		Let \[P=\left(x_1,m_1-x_1,\dots,x_k,m_k-x_k,-\sum_{j=1}^km_j\right)\in\mathbb{P}^{2k}(K),\] and let $S_{\vec{x}}$ be a $(1-\epsilon)$-slice of bad places of $M_K^0\setminus\mathscr{S}_d$ as in (\ref{eqn:Sslice}). Note that \[x_1+(m_1-x_1)+x_2+(m_2-x_2)+\dots+x_k+(m_k-x_k)=\sum_{j=1}^km_j,\] so $P\in\mathcal{H}$ where \[\mathcal{H}:Z_1+\dots+Z_{2k+1}=0\] in $\mathbb{P}^{2k}$. Write \[\eta_v(P)=\log\max\left\{|x_1|_v,|m_1-x_1|_v,\dots,|x_k|_v,|m_k-x_k|_v,\left|-\sum_{j=1}^km_j\right|_v\right\},\] and for $I(P)$ defined as in (\ref{eqn:I}), write \[\textup{rad}_v(P)=\begin{cases} N_v & \textup{if }v\in I(P)\\ 0&\textup{ otherwise.}\end{cases}\] Let $S_{\vec{x}}'$ be the set of places $v\in M_K$ such that at least one of the following mutually exclusive conditions holds: \begin{enumerate}[topsep=5pt]\item Either $|M|_v\ne 1$ for $M:=\sum_{j=1}^km_j$ or $|m_j|_v\ne 1$ for some $1\le j\le k$ \item $v\in M_K^0\setminus\mathscr{S}_d$ is a place of good reduction such that the $v$-adic absolute value of some $\frac{x_1}{m_1},\frac{m_1-x_1}{m_1},\dots,\frac{x_k}{m_k},\frac{m_k-x_k}{m_k}$ is not equal to $1$, and $v$ does not satisfy (i) \item $v\in M_K^\infty\cup\mathscr{S}_d$, and $v$ does not satisfy (i)\item $v\in M_K^0\setminus\mathscr{S}_d$ is a place of bad reduction not in $S_{\vec{x}}$, and $v$ does not satisfy (i).\end{enumerate} Write $S_1$ for the set of places of $S_{\vec{x}}'$ satisfying (i), $S_2$ for the set of places of $S_{\vec{x}}'$ satisfying (ii), and so on. We remark that \begin{equation}\label{eqn:boringplaces}\eta_v(P)=\textup{rad}_v(P)=0 \textup{ for all }v\notin S_{\vec{x}}\cup S_{\vec{x}}'.\end{equation} In what follows, let $N_v=0$ for all $v\in M_K^\infty$. 
		
		For $v\in M_K^0\setminus\mathscr{S}_d$ a place of bad reduction, $\tilde{f}\in L_w[z]$ a monic conjugate of $f$ fixing $0$ (defined over the extension $L_w/K_v$ generated by some fixed point of $f$), $\alpha_i$ the $i$-th degree coefficient of $\tilde{f}$, and $\pi_w\in L_w, \pi_v\in K_v$ uniformizers at $w\mid v$ and $v$ respectively, we have \cite[Lemma 2.1]{Ingram:PCF} \begin{equation}\label{eqn:unif}\begin{split}\lambda_{\textup{crit},v}(f)=\lambda_{\textup{crit},w}(f)=\log\max_{1\le i\le d-1}\{|\alpha_i|_w^{1/(d-i)}\}&\ge \frac{1}{d-1}\log|\pi_w^{-1}|_w\\&\ge\frac{1}{d(d-1)}\log|\pi_v^{-1}|_v.\end{split}\end{equation}

		As $S_{\vec{x}}$ forms a $(1-\epsilon)$-slice of places of $v\in M_K^0\setminus\mathscr{S}_d$, it then follows that \begin{equation}\begin{split}\label{eqn:S4} \sum_{v\in S_4}N_v=\sum_{v\in S_4}r_v\log|\pi_v^{-1}|_v&\le d(d-1)\sum_{v\in S_4}r_v\lambda_{\textup{crit},v}(f)\\&\le d(d-1)\epsilon \sum_{v\in M_K^0\setminus\mathscr{S}_d}r_v\lambda_{\textup{crit},v}(f)\\&\le d(d-1)\epsilon h_{\textup{crit}}(f).\end{split}\end{equation} Therefore 
		\begin{equation}\label{eqn:S1S4}\begin{split}
				\sum_{v\in S_1\cup S_4} r_v\eta_v(P)-(1+\epsilon)\textup{rad}_v(P)\geq& \sum_{v\in S_1\cup S_4} r_v \log |M|_v - (1+\epsilon)N_v
				\\\geq&\sum_{v\in S_1\cap M_K^0} \big( r_v \log |M|_v - (1+\epsilon)r_v\log|\pi_v^{-1}|_v\big) \\&- \sum_{v\in S_4} (1+\epsilon)N_v \\\geq & -\log M - (1+\epsilon)\bigg( \log M + \sum_{j=1}^k \log m_j \bigg) \\&-(1+\epsilon)d(d-1)\epsilon h_{\textup{crit}}(f)
				\\>&-5M-2d(d-1)\epsilon h_{\textup{crit}}(f) .
		\end{split}\end{equation} Here, the second inequality is because 
		$|M|_v= 1$ for $v\not\in S_1$ (hence $v\in S_4$), because $N_v= 0$ for all $v\not\in M_K^0$,
		and because $N_v = r_v \log|\pi_v^{-1}|_v$ for all $v\in M_K^0$.
		The third is by (\ref{eqn:S4}) for the sum over $S_4$, and by the definition of $S_1$ for the rest.
		The fourth is by the fact that $\sum_{j=1}^k \log m_j < \sum_{j=1}^k m_j = M$ and $\epsilon<1$.

		Our analysis of $S_2\cup S_3$, on the other hand, makes key use of the notion of $\epsilon$-adelic goodness; (\ref{eqn:actualgoodness}) along with the fact that $|M|_v=|m_i|_v=1$ for all $1\le i\le k$ and all $v\in S_2\cup S_3$ implies that \begin{equation*}\begin{split}\sum_{v\in S_2\cup S_3}r_v|\eta_v(P)|&\le\sum_{v\in S_2\cup S_3}r_v\log\max\{1,|x_1|_v,|m_1-x_1|_v,\dots,|x_k|,|m_k-x_k|_v\}\\&\le \sum_{v\in S_2\cup S_3}r_v\left(\sum_{i=1}^k\log\max\{1,|x_i|_v\}+\sum_{i=1}^k\log\max\{1,|m_i-x_i|_v\}\right)\\&\le\sum_{v\in S_2\cup S_3}r_v\left(\sum_{i=1}^k|\log|x_i|_v|+\sum_{i=1}^k|\log|m_i-x_i|_v|\right)\\&\le\frac{2k}{d-1}\epsilon h_{\textup{crit}}(f),\end{split}\end{equation*} and by the $K$-rationality of $\frac{x_1}{m_1},\frac{m_1-x_1}{m_1},\dots,\frac{x_k}{m_k},\frac{m_k-x_k}{m_k}$, we further have \[\sum_{v\in S_2\cup S_3}N_v=\sum_{v\in (S_2\cup S_3)\setminus M_K^\infty}N_v+\sum_{v\in M_K^\infty}N_v\le\sum_{v\in (S_2\cup S_3)\setminus M_K^\infty}r_v|\eta_v(P)|_v\le\frac{2k}{d-1}\epsilon h_{\textup{crit}}(f).\] We obtain from this that \begin{equation}\label{eqn:S2S3}\begin{split} \sum_{v\in S_2\cup S_3}r_v\eta_v(P)-(1+\epsilon)\textup{rad}_v(P)&\ge-\frac{1}{d-1}\left(2k\epsilon+(1+\epsilon)2k\epsilon\right)h_{\textup{crit}}(f)\\&\ge-6k\epsilon h_{\textup{crit}}(f),\end{split}\end{equation} again using $\epsilon<1$. Summing (\ref{eqn:S1S4}) and (\ref{eqn:S2S3}) yields \begin{equation}\label{eqn:Sprimecontribution}\begin{split} \sum_{v\in S_{\vec{x}}'}r_v\eta_v(P)-(1+\epsilon)\textup{rad}_v(P)=&\sum_{v\in S_1\cup S_4}r_v\eta_v(P)-(1+\epsilon)\textup{rad}_v(P)\\&+\sum_{v\in S_2\cup S_3}r_v\eta_v(P)-(1+\epsilon)\textup{rad}_v(P)\\>&-5M-\left(2d^2-2d+6k\right)\epsilon h_{\textup{crit}}(f).\end{split}\end{equation} As  $S_{\vec{x}}$ is a $(1-\epsilon)$-slice of places $v\in M_K^0\setminus\mathscr{S}_d$ of bad reduction, and Proposition \ref{prop:Mxi} allows us to assume without loss that $\sum_{v\in M_K^0\setminus\mathscr{S}_d}r_v\lambda_{\textup{crit},v}(f)\ge(1-\epsilon)h_{\textup{crit}}(f)$ and $\epsilon h_{\textup{crit}}(f)\ge\sum_{v\in S_1}r_v\lambda_{\textup{crit},v}(f)$, \begin{equation*}\label{eqn:Sprime}\begin{split}\sum_{v\in S_{\vec{x}}'}r_v\lambda_{\textup{crit},v}(f)&\le h_{\textup{crit}}(f)-\sum_{v\in S_{\vec{x}}}r_v\lambda_{\textup{crit},v}(f)+\sum_{v\in S_1}r_v\lambda_{\textup{crit},v}(f)\\&\le h_{\textup{crit}}(f)-(1-\epsilon)\sum_{v\in M_K^0\setminus\mathscr{S}_d}r_v\lambda_{\textup{crit},v}(f)+\sum_{v\in S_1}r_v\lambda_{\textup{crit},v}(f)\\&\le h_{\textup{crit}}(f)-(1-\epsilon)^2h_{\textup{crit}}(f)+\sum_{v\in S_1}r_v\lambda_{\textup{crit},v}(f)\\&\le2\epsilon h_{\textup{crit}}(f)+\epsilon h_{\textup{crit}}(f)\\&=3\epsilon h_{\textup{crit}}(f).\end{split}\end{equation*} By (\ref{eqn:finalquotientbd}), we thus have \begin{equation}\begin{split}\label{eqn:SnotSprime}\sum_{v\in S_{\vec{x}}\setminus S_{\vec{x}}'}r_v\eta_v(P)-(1+\epsilon)\textup{rad}_v(P)\ge&\sum_{v\in S_{\vec{x}}\setminus S_{\vec{x}}'}r_v\lambda_{\textup{crit},v}(f)-\sum_{v\in S_{\vec{x}}\setminus S_{\vec{x}}'}\epsilon N_v\\ \ge&\sum_{v\in S_{\vec{x}}}r_v\lambda_{\textup{crit},v}(f)-\sum_{v\in S_{\vec{x}}'}r_v\lambda_{\textup{crit},v}(f)-\sum_{v\in S_{\vec{x}}\setminus S_{\vec{x}}'}\epsilon N_v\\ \ge&(1-\epsilon)^2h_{\textup{crit}}(f)-3\epsilon h_{\textup{crit}}(f)\\&-\sum_{v\in S_{\vec{x}}\setminus S_{\vec{x}}'}\epsilon d(d-1)\lambda_{\textup{crit},v}(f)\\ \ge&(1-(5+d^2-d)\epsilon)h_{\textup{crit}}(f),\end{split}\end{equation} where the third inequality again uses (\ref{eqn:Nvlambda}). Finally, by (\ref{eqn:boringplaces}), we clearly have  \begin{equation}\label{eqn:leftover}\begin{split}\sum_{v\notin S_{\vec{x}}\cup S_{\vec{x}}'}r_v\eta_v(P)-(1+\epsilon)\textup{rad}_v(P)=0.\end{split}\end{equation} 
		
		Summing (\ref{eqn:Sprimecontribution}), (\ref{eqn:SnotSprime}), and (\ref{eqn:leftover}) gives \begin{equation}\label{eqn:lastline}\begin{split}\sum_{v\in M_K}r_v\eta_v(P)-(1+\epsilon)\textup{rad}_v(P)=&\sum_{v\in S_{\vec{x}}'}r_v\eta_v(P)-(1+\epsilon)\textup{rad}_v(P)\\&+\sum_{v\in S_{\vec{x}}\setminus S_{\vec{x}}'}r_v\eta_v(P)-(1+\epsilon)\textup{rad}_v(P)\\&+\sum_{v\notin S_{\vec{x}}\cup S_{\vec{x}}'}r_v\eta_v(P)-(1+\epsilon)\textup{rad}_v(P)\\>&-5M-\left(2d^2-2d+6k\right)\epsilon h_{\textup{crit}}(f)\\&+(1-(5+d^2-d)\epsilon)h_{\textup{crit}}(f)\\=&-5M+(1-(6k+3d^2-3d+5)\epsilon) h_{\textup{crit}}(f).\end{split}\end{equation} In other words, \begin{equation*}\label{eqn:tada}\begin{split} h(P)-(1+\epsilon)\text{rad}(P)&> -5M+\left(1-\left(6k+3d^2-3d+5\right)\epsilon\right)h_{\textup{crit}}(f)\\&=-5M+\left(1-\left(6\left\lceil \frac{\log(\epsilon^2/2)}{C(d)} \right\rceil+3d^2-3d+5\right)\epsilon\right)h_{\textup{crit}}(f).\end{split}\end{equation*} As $\lim_{\epsilon\to0}\left\lceil \frac{\log(\epsilon^2/2)}{C(d)} \right\rceil\cdot\epsilon=0$, we obtain a contradiction of Conjecture \ref{conj:abcd} whenever $\epsilon\ll_{d}1$ and $h_{\textup{crit}}(f)\ge M'$ for some $M'=M'(\epsilon,M)$. Finally, if $h_{\textup{crit}}(f)<M'$, then as $f$ has been assumed to be non-isotrivial if $K$ is a function field, we may apply Proposition \ref{prop:Mxi} to obtain the desired result in these cases.\end{proof}

	\begin{proof}[Proof of Theorems \ref{thm:UBCpolys} and \ref{thm:dynLangpolys}] Theorem \ref{thm:UBCpolys} is immediate from Theorem \ref{thm:combined}, as preperiodic points all have canonical height $0$. Now let all notation be as in Theorem \ref{thm:combined}, except write $\kappa'$ in lieu of $\kappa$. Suppose first that either $K$ is a number field or that $K$ is a function field and $f$ is not isotrivial. Then letting $\kappa=d^{-B}\kappa'$, the fact that $\hat{h}_f(f(P))=d\hat{h}_f(P)$ for all $P\in\overline{K}$ implies together with Theorem \ref{thm:combined} that each $P\in K$ with $\hat{h}_f(P)\ne0$ satisfies $\hat{h}_f(P)\ge\kappa\max\{1,h_{\textup{crit}}(f)\}$, proving Theorem \ref{thm:dynLangpolys} in this case. Now suppose that $K$ is a function field and that $f$ is isotrivial. Then Lemma \ref{lem:isotrivLang} at once yields Theorem \ref{thm:dynLangpolys}, noting that in this case we have $h_{\textup{crit}}(f)=0$ and hence $\max\{1,h_{\textup{crit}}(f)\}=1$. \end{proof}

\end{document}